\documentclass[a4paper,10pt]{amsart}

\usepackage{amsthm,amssymb}
\usepackage{stmaryrd}
\usepackage[T1]{fontenc}
\usepackage[utf8]{inputenc}
\usepackage{enumitem}
\usepackage{mathrsfs}
\usepackage{dsfont}
\usepackage{mathtools}
\usepackage{scalerel,stackengine}
\usepackage{nicefrac}
\usepackage[margin=3.5cm]{geometry}
\usepackage{tikz}
\usepackage{comment}
\usepackage[all,cmtip]{xy}
\usetikzlibrary{positioning}

\usepackage{natbib}
\usepackage[hypertexnames=false]{hyperref}
\hypersetup{
    colorlinks,
    linkcolor={red!90!black},
    citecolor={green!70!black},
    urlcolor={blue!80!black}
}

\numberwithin{equation}{section}

\makeatletter
\DeclareSymbolFont{lettersA}{U}{txmia}{m}{it}
\SetSymbolFont{lettersA}{bold}{U}{txmia}{bx}{it}
\DeclareFontSubstitution{U}{txmia}{m}{it}

\DeclareMathSymbol{\m@thbbch@rA}{\mathord}{lettersA}{129}
\DeclareMathSymbol{\m@thbbch@rB}{\mathord}{lettersA}{130}
\DeclareMathSymbol{\m@thbbch@rC}{\mathord}{lettersA}{131}
\DeclareMathSymbol{\m@thbbch@rD}{\mathord}{lettersA}{132}
\DeclareMathSymbol{\m@thbbch@rE}{\mathord}{lettersA}{133}
\DeclareMathSymbol{\m@thbbch@rF}{\mathord}{lettersA}{134}
\DeclareMathSymbol{\m@thbbch@rG}{\mathord}{lettersA}{135}
\DeclareMathSymbol{\m@thbbch@rH}{\mathord}{lettersA}{136}
\DeclareMathSymbol{\m@thbbch@rI}{\mathord}{lettersA}{137}
\DeclareMathSymbol{\m@thbbch@rJ}{\mathord}{lettersA}{138}
\DeclareMathSymbol{\m@thbbch@rK}{\mathord}{lettersA}{139}
\DeclareMathSymbol{\m@thbbch@rL}{\mathord}{lettersA}{140}
\DeclareMathSymbol{\m@thbbch@rM}{\mathord}{lettersA}{141}
\DeclareMathSymbol{\m@thbbch@rN}{\mathord}{lettersA}{142}
\DeclareMathSymbol{\m@thbbch@rO}{\mathord}{lettersA}{143}
\DeclareMathSymbol{\m@thbbch@rP}{\mathord}{lettersA}{144}
\DeclareMathSymbol{\m@thbbch@rQ}{\mathord}{lettersA}{145}
\DeclareMathSymbol{\m@thbbch@rR}{\mathord}{lettersA}{146}
\DeclareMathSymbol{\m@thbbch@rS}{\mathord}{lettersA}{147}
\DeclareMathSymbol{\m@thbbch@rT}{\mathord}{lettersA}{148}
\DeclareMathSymbol{\m@thbbch@rU}{\mathord}{lettersA}{149}
\DeclareMathSymbol{\m@thbbch@rV}{\mathord}{lettersA}{150}
\DeclareMathSymbol{\m@thbbch@rW}{\mathord}{lettersA}{151}
\DeclareMathSymbol{\m@thbbch@rX}{\mathord}{lettersA}{152}
\DeclareMathSymbol{\m@thbbch@rY}{\mathord}{lettersA}{153}
\DeclareMathSymbol{\m@thbbch@rZ}{\mathord}{lettersA}{154}

\long\def\DoLongFutureLet #1#2#3#4{%
   \def\@FutureLetDecide{#1#2\@FutureLetToken
      \def\@FutureLetNext{#3}\else
      \def\@FutureLetNext{#4}\fi\@FutureLetNext}
   \futurelet\@FutureLetToken\@FutureLetDecide}
\def\DoFutureLet #1#2#3#4{\DoLongFutureLet{#1}{#2}{#3}{#4}}
\def\@EachCharacter{\DoFutureLet{\ifx}{\@EndEachCharacter}%
   {\@EachCharacterDone}{\@PickUpTheCharacter}}
\def\m@keCharacter#1{\csname\F@ntPrefix#1\endcsname}
\def\@PickUpTheCharacter#1{\m@keCharacter{#1}\@EachCharacter}
\def\@EachCharacterDone \@EndEachCharacter{}

\DeclareRobustCommand*{\varmathbb}[1]{\gdef\F@ntPrefix{m@thbbch@r}%
  \@EachCharacter #1\@EndEachCharacter}
\makeatother

\newtheorem{theorem}{Theorem}[section]
\newtheorem{lemma}[theorem]{Lemma}
\newtheorem{proposition}[theorem]{Proposition}
\newtheorem{corollary}[theorem]{Corollary}

\newtheoremstyle{mytheoremstyle} 
    {1em plus .2em minus .1em}                    
    {1em plus .2em minus .1em}                    
    {\rmfamily}                   
    {}                           
    {\bfseries}                   
    {.}                          
    {.5em}                       
    {}  

\theoremstyle{mytheoremstyle}

\newtheorem{definition}[theorem]{Definition}
\newtheorem{example}[theorem]{Example}
\newtheorem{remark}[theorem]{Remark}

\newcommand\iffdef{\;\mathrel{\mathord{:}\mathord{\longleftrightarrow}}\;}

\newcommand\defeq{\coloneqq} 

\newcommand{\QED}{\hfill$\dashv$} 

\renewcommand{\diamond}{\Diamond}
\newcommand{\vD}{\vdash_{\Diamond}} 
\newcommand{\Dv}{\Diamond_{\vdash}} 
\DeclareMathOperator{\power}{\mathcal{P}} 
\DeclareMathOperator{\closedne}{\mathscr{C}_{0}} 
\DeclareMathOperator{\Ult}{Ul} 
\DeclareMathOperator{\Uf}{\mathsf{Uf}} 

\renewcommand{\mid}{:}
\newcommand{\con}{\mathrel{\mathsf{C}}}
\DeclareMathOperator{\RC}{RC} 
\DeclareMathOperator{\Int}{Int} 
\newcommand{\Cl}{\mathop{\mathrm{Cl}}}

\newcommand{\CO}{\mathscr{C\!O}} 

\newcommand{\RD}{R_{\Diamond}} 

\newcommand{\frF}{\mathfrak{F}}

\newcommand{\Nat}{\varmathbb{N}} 

\newcommand{\tand}{\text{ and }}

\newcommand{\qtiff}{\quad\text{iff}\quad}

\renewcommand{\varnothing}{\emptyset}

{\medskip\color{magenta}{\noindent RG: \  }}{\hfill{$\square$}\color{black}\par\medskip}

\title{Extended Contact Algebras: Algebraic Analysis and Duality Theory}

\author[]{Rafa\l\ Gruszczy\'nski, Paula Mench\'on, and William Zuluaga}\thanks{This research was funded by (a) the National Science Centre (Poland), grant number~2020/39/B/HS1/00216 and (b) the MOSAIC project (EU H2020-MSCA-RISE-2020 Project 101007627).}

\date{}

\address{Rafa\l\ Gruszczy\'nski, \textsc{Orcid:} 0000-0002-3379-0577\\
Department of Logic\\
Nicolaus Copernicus University in Toru\'n\\
Poland}

\email{gruszka@umk.pl}

\address{Paula Mench\'on, \textsc{Orcid:} 0000-0002-9395-107X\\
Department of Logic\\
Nicolaus Copernicus University in Toru\'n\\
Poland\\
and
University of the Center of the Buenos Aires Province, Tandil (Unicen)\\
Argentina
}

\email{mpmenchon@nucompa.exa.unicen.edu.ar}

\address{William Zuluaga, \textsc{Orcid}: 0000-0002-8798-9493\\ Conicet and
University of the Center of the Buenos Aires Province, Tandil (Unicen)\\
Argentina}
\email{wizubo@gmail.com}

\begin{document}

\begin{abstract}
The ternary extended contact relation studied in \citep{Ivanova-ECAATC} is a more expressive counterpart of the standard, well-known binary contact relation. Boolean algebras expanded with this relation are called Extended Contact Algebras (ECAs).
In this work, we take an algebraic perspective on ECAs, interpreting the ternary relation as a form of entailment. To this end, we introduce the class of \emph{Pseudo-Inference Algebras}, i.e., Boolean algebras
equipped with a monotone ternary operator, additive in its first two coordinates, and satisfying additional
equations that capture the logical character of the extended contact. We show that the subclass of relational Pseudo-Inference Algebras corresponds precisely to ECAs and generates the variety of \emph{strict PSI-Algebras}, which is a discriminator variety. Furthermore,  building on the duality developed in \citep{Celani2009},
we introduce \emph{descriptive PSI-frames} and establish three interrelated dualities that differ in their morphisms while sharing the same class of topological objects. The resulting framework provides a relational semantics for
PSI-Algebras, connecting spatial and logical notions within
a categorical and topological setting.

\smallskip


\end{abstract}

\maketitle

\section{Introduction}\label{Section 1}

An Extended Contact Algebra
(ECA, for short) is a structure obtained by expanding a~Boolean algebra $A$ with a ternary relation among the elements of the domain of~$A$.
According to \cite{Vakarelov-AMBOSQ}, such structures were first considered in \citep{Ivanova2015}, and their theory was later developed in \citep{Ivanova-ECAATC}. The idea was motivated by certain limitations related to the well-known Boolean Contact Algebras (BCAs), i.e., Boolean algebras expanded with a binary relation $\con$ of contact. BCAs were introduced and studied in the early 2000s by \cite{Stell-BCAANATRCC}, \cite{Duntsch-Winter-CBCA,Duntsch-et-al-RTBCA}, and \cite{Dimov-et-al-CARBTSPA1}. They were intended as a framework for representing and examining the nearness of regions in space in an abstract algebraic setting.\footnote{Contact relations themselves were put forward in the 1920s by \cite{deLaguna-PLS} and \cite{Whitehead-CN,Whitehead-PR}, and later the interest in them was revived through the works of \cite{Clarke-CIBC} and \cite{Biacino-Gerla-CS}, and the development of the Region Connection Calculus by \cite{Randell-et-al-ASLBORAC}. BCAs are an algebraic realization of the ideas of the aforementioned authors. From a certain point of view, they can also be seen as an abstract approach to proximity spaces in the sense of \cite{Thron-PSAG} (see, e.g., \citealp{Duntsch-et-al-RBTODSAPA}).}

The idea of BCAs proved both fruitful and useful, giving insight into the algebraic aspects of nearness. Among other things, it turned out that BCAs have a considerable expressive power that lets them capture various topological properties by means of the contact relation. However, unsurprisingly, natural properties were discovered that could not be grasped within the framework of the contact relation. One of these was the property of interior connectedness of regular closed subsets of topological spaces. Since the topological property of connectedness of regions---particularly important for mereotopology~\citep{Pratt-Schoop-EIPPM}---can be naturally expressed by means of the contact relation, it could be expected that a similar situation holds for the property of interior connectedness. Contrary to this expectation, Tatyana Ivanova (\citeyear{Ivanova-ECAATC}) proved that the property cannot be formulated by means of contact.\footnote{See also \citep{Gruszczynski-Menchon-ONDOICVCR}.} At the same time, she put forward the idea of a ternary relation of \emph{extended contact},
which strengthens the ordinary contact and allows for a definition of the internal connectedness
of regular closed sets.

Formally, an \emph{Extended Contact Algebra} is a structure of the form
\[
(A,\vee,\wedge,\neg,0,1,\vdash),
\]
where $(A,\vee,\wedge,\neg,0,1)$ is a non-degenerate Boolean algebra and
$\vdash$ is a ternary relation on $A$ such that, for all $a,b,d,e,f\in A$, the following conditions hold  ($\leq$ is the standard binary Boolean order relation):
\begin{align*}
  \label{ExtCA0}\tag{ExtCA0}  &\text{If $(a,b)\vdash f$, then $(a\vee d,\,b)\vdash d\vee f$}.\\
  \label{ExtCA1}\tag{ExtCA1}  &\text{If $(a,b)\vdash d$, $(a,b)\vdash e$ and $(d,e)\vdash f$, then $(a,b)\vdash f$}.\\
  \label{ExtCA2}\tag{ExtCA2}  &\text{If $a\le f$, then $(a,b)\vdash f$}.\\
  \label{ExtCA3}\tag{ExtCA3}  &\text{If $(a,b)\vdash f$, then $a\wedge b\le f$}.\\
  \label{ExtCA4}\tag{ExtCA4}  &\text{If $(a,b)\vdash f$, then $(b,a)\vdash f$}.
\end{align*}

It is noteworthy that the relation $\vdash$ has a \emph{logical character} (in a certain sense), since it behaves analogously to a consequence relation,
capturing a form of inference or entailment between regions within the underlying Boolean framework.  Moreover, this relational perspective---and its generalization from \citep{Vakarelov-AMBOSQ}
in the form of \emph{sequent algebras}---offers a new view of the region-based approach
to theories of space, establishing a conceptual bridge between the spatial notion of
\emph{nearness} and the logical notion of \emph{entailment}.

Interpreting $(a,b)\vdash c$ as an entailment from the premises
$a,b$ to the conclusion $c$ motivates encoding its failure
by a Boolean-valued operation. Specifically, we define
$\Diamond_{\vdash}(a,b,c)\defeq\neg\chi_{\vdash}(a,b,\neg c)$,
where $\chi_{\vdash}$ is the characteristic function of
$\vdash$. Thus, $(a,b)\vdash c$ holds exactly when
$\Diamond_{\vdash}(a,b,\neg c)=0$.
This translation turns the rules governing the entailment
relation into algebraic conditions on the operation.
For example, the condition \eqref{ExtCA1} (which may be interpreted as a version of the cut rule) gives rise
to an equation of the form of \eqref{PI1} (see further, Definition~\ref{def:pseudo-inference algebra}), and so on. We then use these conditions to define Pseudo-Inference Algebras
(PSI-Algebras), abandoning the restriction of values to $0$ and $1$. Their class restricted to \emph{relational} PSI-Algebras
recovers precisely the original ECAs.

The paper is organized as follows. Section~\ref{Section 2} recalls ECAs and presents a simplified, yet equivalent, axiomatization of the system originally introduced in~\citep{Ivanova-ECAATC}. This formulation yields a more transparent algebraic characterization of the associated operator. Section~\ref{Section 3} introduces the variety of Pseudo-Inference Algebras and establishes a one-to-one correspondence between its subclass of relational PSI-algebras and the class of ECAs. In Section~\ref{Section 4}, we investigate the class of relational PSI-algebras in detail, showing that it generates the subvariety of \emph{strict} PSI-Algebras, defined by the additional equations \eqref{R1}--\eqref{R2} and \eqref{S}, and that this subvariety forms a~discriminator variety.  Section~\ref{Section 5} introduces descriptive PSI-frames and establishes the representation of 2-additive 3-BAMOs (Boolean algebras with a normal ternary monotonic operator that is additive in its first two coordinates), building on the framework of \citep{Celani2009}. Section~\ref{sec:psi-spaces} introduces PSI-spaces and establishes the correspondence between their defining conditions and the axioms of PSI-Algebras. Finally, Section~\ref{sec:morphisms-dualities} studies morphisms and establishes three interrelated dualities between PSI-Algebras and PSI-spaces (gathered in various categories with the same objects but different classes of morphisms), together with their restrictions to ECAs and total PSI-spaces.  Throughout the paper, we assume the reader is familiar with standard notions from category theory and universal algebra as presented in \citep{ML1978} and \citep{BS2011}, respectively.

\section{Extended Contact Algebras}\label{Section 2}

In our approach to ECAs, we will use a slightly simplified---albeit equivalent---set of constraints from \citep{Ivanova-ECAATC}. Thanks to this, on the side of the operator associated with the relation, we obtain a relatively simple and easy-to-handle set of algebraic postulates.

Consider a structure of the form $(\mathbf{A}, \vdash)$,
where $\mathbf{A}=(A,\vee,\wedge,\neg,0,1)$ is a non-degenerate Boolean algebra (whose elements are called \emph{regions}) and $\vdash$ is a ternary relation on $A$.
We retain \eqref{ExtCA0}, \eqref{ExtCA1} and \eqref{ExtCA4},
and replace \eqref{ExtCA2} and \eqref{ExtCA3} by the following universal
conditions:
\begin{align}
\label{EC2}\tag{ExtCA2$'$} &(a,b)\vdash a\vee f,\\
\label{EC3}\tag{ExtCA3$'$} &\text{If $(a,a)\vdash f$, then $a\leq f$}.
\end{align}
Remark~\ref{prop:EC1-EC2} and Proposition~\ref{prop:EC3} below establish the equivalence of the two axiomatizations.

The ECA axioms can be interpreted as expressing elementary facts about entailment (in the very special case when we have no more than two premises). To see this, let us interpret the Boolean order $x\leq y$ as $x$ \emph{carries more information than or the same information as} $y$. According to \eqref{ExtCA0}, if $a$ and $b$ entail $f$, then we can weaken one of the premises and the conclusion simultaneously. \eqref{ExtCA1} is an instance of the cut rule: if $a$ and $b$ entail both $d$ and $e$, which entail $f$, then we can eliminate $d$ and $e$. According to \eqref{EC2}, we can always weaken the conclusion, and according to \eqref{EC3}, the conclusion never carries more information than is contained in both the premises. Finally, \eqref{ExtCA4} says that the order of premises is not relevant for the entailment.

\begin{proposition}
If \eqref{ExtCA1} and \eqref{EC2} hold in a structure $(\mathbf{A}, \vdash)$, then we have that
\begin{align}
    \text{If $(a,b)\vdash c$ and $c\leq f$, then $(a,b)\vdash f$.}\label{eq:ECA-weaker-right}
\end{align}
\end{proposition}
\begin{proof}
    Indeed, suppose that $(a,b)\vdash c$ and $c\leq f$. By \eqref{EC2} we have that $(c,c)\vdash c\vee f$, so $(c,c)\vdash f$. By \eqref{ExtCA1} it must be the case that $(a,b)\vdash f$.
\end{proof}
\begin{remark}\label{prop:EC1-EC2}
    The axiom \eqref{EC2} is equivalent to \eqref{ExtCA2}. It is clear that \eqref{ExtCA2} entails \eqref{EC2}. For the  other direction, if $a\leq f$, then we have that $a\vee f=f$. But by \eqref{EC2} it must be the case that $(a,b)\vdash a\vee f$, so $(a,b)\vdash f$.
\end{remark}

By \eqref{ExtCA2}, $a\leq f$ implies $(a,a)\vdash f$. Together with \eqref{EC3}, this yields the equivalence
\begin{equation}
    (a,a)\vdash f\qtiff a\leq f\,.
\end{equation}

\begin{proposition}
If \eqref{ExtCA1}, \eqref{EC2} and \eqref{ExtCA4} hold in a structure $(\mathbf{A}, \vdash)$, it is the case that
\begin{equation}
        \text{If $(a,b)\vdash c$ and $f\leq a$, then $(f,b)\vdash c$.}\label{eq:ECA-stronger-left}
\end{equation}
\end{proposition}
\begin{proof}
    Suppose that $(a,b)\vdash c$ and $f\leq a$. By \eqref{EC2} we obtain $(f,b)\vdash a$. By \eqref{EC2} and \eqref{ExtCA4}, we obtain $(f,b)\vdash b$. From these, by \eqref{ExtCA1}, it follows that $(f,b)\vdash c$.
\end{proof}

\begin{proposition}\label{prop:EC3}
    Assuming \eqref{ExtCA1}, \eqref{EC2} and \eqref{ExtCA4},
condition \eqref{EC3} is equivalent to \eqref{ExtCA3}.
\end{proposition}
\begin{proof}
    One direction is immediate. For the other, assume that \eqref{EC3} holds and that $(a,b)\vdash f$. Using \eqref{eq:ECA-stronger-left} and \eqref{ExtCA4} we conclude that $(a\wedge b,a\wedge b)\vdash f$, and so $a\wedge b\leq f$, by the assumption.
\end{proof}

Unlike \cite{Ivanova-ECAATC}, who includes
binary contact and interior connectedness as predicates, we work
in the reduced signature consisting of the Boolean operations
and the ternary relation $\vdash$. Both, the contact and the interior
connectedness are treated as notions definable by $\vdash$,
rather than as additional primitive predicates.

The contact relation can be defined by means of $\vdash$ as follows
\[
a\con b\iffdef (a,b)\nvdash 0,
\]
but $\vdash$ cannot be captured in the standard language of BCAs (see \citealp[Propositions 2.1 and 3.1]{Ivanova-ECAATC}). As in the case of the contact relation, the standard interpretation of extended contact is topological; that is for a topological space $(X,\tau)$, any subalgebra $S$ of the Boolean algebra $\RC(\tau)$ of regular closed subsets of $X$  with the ternary relation $\vdash_{\tau}$ given by
\[
(A,B)\vdash_{\tau} C\iffdef A\cap B\subseteq C
\]
is an Extended Contact Algebra. The operation $\cap$ on the right side of the equivalence is the usual set-theoretical intersection, which, in general, differs from the meet operation of $\RC(\tau)$, which for any $A,B$ is given by $A\wedge B\defeq\Cl\Int(A\cap B)$, the closure of the interior of the intersection. \citet[Theorem 6.4]{Ivanova-ECAATC} proves that every Extended Contact Algebra can be embedded into the regular closed algebra of a compact, $T_0$, semi-regular space. Different, so-called relational, representations of (finite) ECAs can be found in \citep{Balbiani-et-al-RRTFECA}.

\begin{proposition}[{\citealp[Proposition 3]{Balbiani-et-al-RRTFECA}}]\label{Props ECA}
    Let $(\mathbf{A}, \vdash)$ be an Extended Contact Algebra. For all \(a, b, c, d \in A\), the following conditions hold:
\begin{align}
    &(0, 1) \vdash a.\label{eq:BI-1}\\
    &\text{If \((a, b) \vdash c\) and \((d, b) \vdash c\), then \((a \vee d, b) \vdash c\).}\label{eq:BI-2}\\
    &\text{If \((a, b) \vdash c\) and \(d \leq a\), then \((d, b) \vdash c\).}\label{eq:BI-3}\\
    &\text{If $(a,b)\vdash c$ and $c\leq d$, then $(a,b)\vdash d$.}\label{eq:BI-4}
\end{align}
\end{proposition}

We conclude this section by presenting some properties of the characteristic function of~$\vdash$. In a sense that will be made explicit in the following section, this result reveals the algebraic behavior that underlies ECAs.

\begin{lemma}\label{EC as 3BAMO}
 Let $(\mathbf{A},\vdash)$ be an ECA and let $\chi_{\vdash}$ be the characteristic function of $\vdash$. Then, for every $a,b,c,x\in A$ the following hold:
 \begin{gather}
\chi_{\vdash}(0,b,c)=\chi_{\vdash}(a,0,c)=\chi_{\vdash}(a,b,1)=1\,,\label{eq:ch-1}\\
\chi_{\vdash}(a\vee x,b,c)=\chi_{\vdash}(a,b,c)\wedge \chi_{\vdash}(x,b,c)\,,\label{eq:ch-2}\\
\chi_{\vdash}(a,b\vee x,c)=\chi_{\vdash}(a,b,c)\wedge \chi_{\vdash}(a,x,c)\,,\label{eq:ch-3}\\
\chi_{\vdash}(a,b,c\wedge x)\leq \chi_{\vdash}(a,b,c)\wedge \chi_{\vdash}(a,b,x)\,.\label{eq:ch-4}
 \end{gather}
\end{lemma}
\begin{proof}
    \eqref{eq:ch-1} Since $0\leq c$, we obtain $(0,b)\vdash c$ by \eqref{EC2}, so $\chi_{\vdash}(0,b,c)=1$. Similarly, by \eqref{EC2} and \eqref{ExtCA4}, we obtain that $\chi_{\vdash}(a,0,c)=1$. Finally, because $a\leq 1$, from \eqref{EC2} again we obtain $(a,b)\vdash 1$, and thus $\chi_{\vdash}(a,b,1)=1$.

\smallskip

    \eqref{eq:ch-2} To prove the inequality $\leq$, let $\chi_{\vdash}(a\vee x,b,c)=1$. Thus, $(a\vee x,b)\vdash c$, and from this by \eqref{eq:BI-3} applied twice we obtain that  $(a,b)\vdash c$ and $(x,b)\vdash c$. In consequence, $\chi_{\vdash}(a,b,c)=1$ and $\chi_{\vdash}(x,b,c)=1$. Thus, $\chi_{\vdash}(a\vee x,b,c)\leq \chi_{\vdash}(a,b,c)\wedge \chi_{\vdash}(x,b,c)$.

    \smallskip

    To prove the reverse inequality, suppose that $\chi_{\vdash}(a,b,c)=\chi_{\vdash}(x,b,c)=1$. Then $(a,b)\vdash c$ and $(x,b)\vdash c$. \eqref{eq:BI-2} implies that $(a\vee x,b)\vdash c$, so $\chi_{\vdash}(a\vee x,b,c)=1$. Therefore $\chi_{\vdash}(a,b,c)\wedge \chi_{\vdash}(x,b,c)\leq \chi_{\vdash}(a\vee x,b,c)$. Combining both inequalities, we obtain the equality.

\smallskip

\eqref{eq:ch-3} This follows from \eqref{eq:ch-2} by the symmetry of $\chi_{\vdash}$ in its first two coordinates, which follows from \eqref{ExtCA4}.

    \eqref{eq:ch-4} Suppose that $\chi_{\vdash}(a,b,c\wedge x)=1$, i.e., $(a,b)\vdash c\wedge x$. By \eqref{eq:BI-4} we obtain $(a,b)\vdash c$ and $(a,b)\vdash x$. Therefore, $\chi_{\vdash}(a,b,c)\wedge \chi_{\vdash}(a,b,x)=1$. Thus, $\chi_{\vdash}(a,b,c\wedge x)\leq \chi_{\vdash}(a,b,c)\wedge \chi_{\vdash}(a,b,x)$.

This concludes the proof.
\end{proof}

\section{Pseudo-Inference Algebras}\label{Section 3}

In this section, we introduce the variety of Pseudo-Inference Algebras and show that there is a one-to-one correspondence between its subclass of \emph{relational PSI-algebras} and the class of ECAs.

We begin by recalling the framework of Boolean algebras with normal $n$-ary
monotonic operators (BAMOs) from \citep[Definition 1]{Celani2009}. A BAMO is a pair $(\mathbf{A},\Diamond)$, where $\mathbf{A}$ is a Boolean algebra and $\Diamond\colon A^n\to A$ is monotonic in each coordinate and normal. The former means that
\[
\Diamond(a_1,\ldots,a_n)\leq\Diamond(b_1,\ldots,b_n)
\quad\text{whenever }a_i\leq b_i\text{ for every }i\in\{1,\ldots,n\},
\]
and the latter that the value of $\Diamond$ is $0$
whenever at least one argument is $0$. In the sequel, we consider $\Diamond$ to be a ternary operator which is additive in the first two coordinates.

\begin{definition}
    A \emph{2-additive 3-BAMO} is a Boolean algebra
$\mathbf{A}=(A,\vee,\wedge,\neg,0,1)$ equipped with a normal
ternary monotonic operator $\Diamond$ that preserves binary joins
in its first and second coordinates. Equivalently, the pair
$(\mathbf{A},\Diamond)$ satisfies the following conditions for
all $a,b,c,x\in A$:
\begin{gather}
\Diamond(0,b,c)=\Diamond(a,0,c)=\Diamond(a,b,0)=0\,,\label{MO1}\tag{MO1}\\
\Diamond(a\vee x,b,c)=\Diamond(a,b,c)\vee \Diamond(x,b,c)\,,\label{MO2}\tag{MO2}\\
\Diamond(a,b\vee x,c)=\Diamond(a,b,c)\vee \Diamond(a,x,c)\,,\label{MO3}\tag{MO3}\\
\Diamond(a,b,c\vee x)\geq \Diamond(a,b,c)\vee \Diamond(a,b,x)\,.\label{MO4}\tag{MO4}
\end{gather}\QED
\end{definition}
Here \eqref{MO1} expresses normality, \eqref{MO2} and \eqref{MO3}
express additivity in the first two coordinates, and \eqref{MO4}
expresses monotonicity in the third coordinate. Additivity also
implies monotonicity in the first two coordinates. The role of the term
``2-additive'' is to indicate the additivity restricted to the first and the second coordinates. No such condition is imposed on the third coordinate.

\begin{example}\label{example 3BAMO}
The following is an example of a 2-additive 3-BAMO defined on the four-element Boolean algebra $\{0,a,b,1\}$, with $b=\neg a$. We only exhibit
the values of the operation $\Diamond$ that are different from zero.
In all other cases, $\Diamond(x,y,z) = 0$.

\[
\begin{aligned}
\Diamond(1,1,1) &= 1, & \Diamond(1,1,a) &= a, & \Diamond(1,1,b) &= 1,\\
\Diamond(1,a,1) &= 1, & \Diamond(1,a,a) &= a, &
\Diamond(1,b,1) &= 1, \\ \Diamond(1,b,b) &= 1,&
\Diamond(a,1,1) &= 1, & \Diamond(a,1,a) &= a,\\
\Diamond(a,a,1) &= 1, & \Diamond(a,a,a) &= a,&
\Diamond(b,1,1) &= 1, \\ \Diamond(b,1,b) &= 1,&
\Diamond(b,b,1) &= 1, & \Diamond(b,b,b) &= 1.
\end{aligned}
\]
A direct check of the table, including the values stipulated to be zero, shows that \eqref{MO1}--\eqref{MO4} hold. However, the operation is not additive in its third coordinate, since
\[
\Diamond(a,a,a)\vee\Diamond(a,a,b)=a<1=\Diamond(a,a,1).
\]
\QED
\end{example}

To obtain the variety of Pseudo-Inference Algebras we extend the 2-additive 3-BAMO axioms with additional postulates corresponding to axioms \eqref{ExtCA1}, \eqref{EC2}, \eqref{EC3} and \eqref{ExtCA4}; the counterpart of \eqref{ExtCA0} turns out to be derivable (see the proof of Lemma~\ref{extended contact are relational pseudo-iinference algebras}).

\begin{definition}\label{def:pseudo-inference algebra}
    A 2-additive 3-BAMO $(\mathbf{A}, \Diamond)$ is called a \emph{Pseudo-Inference Algebra} (PSI-Algebra for short) if, for every $a,b,d,e,f\in A$, the following conditions hold:
    \begin{gather}
    \Diamond(a,b,f)\leq \Diamond(a,b,\neg d)\vee \Diamond(a,b,\neg e) \vee  \Diamond(d,e,f)\,,\tag{PI1}\label{PI1}\\
    \Diamond(a,b, \neg a)=0\,,\tag{PI2}\label{PI2}\\
    a\wedge f\leq \Diamond(a,a,f)\,,\tag{PI3}\label{PI3}\\
    \Diamond(a,b,f)\leq \Diamond(b,a,f)\,.\tag{PI4}\label{PI4}
    \end{gather}\QED
\end{definition}

The proof of the following proposition is a straightforward consequence of \eqref{MO3} and \eqref{MO4}. The details are left to the reader.

\begin{proposition}\label{quasiecuation PI2}
    Let $(\mathbf{A},\Diamond)$ be a 2-additive 3-BAMO. Then \eqref{PI2} is equivalent to the equation
    \begin{equation}
        \Diamond(a,1,\neg a)=0
    \end{equation}
    and to each of the two quasi-equations below
    \begin{align}
        a\wedge f &{}= 0 \Rightarrow \Diamond(a,b,f)=0\,,\\
        a\wedge f &{}= 0 \Rightarrow \Diamond(a,1,f)=0\,.
    \end{align}
\end{proposition}

\begin{example}\label{smallest Diamond}
Every Boolean algebra $\mathbf{A}$
can be turned into a PSI-Algebra by defining
\[
\Diamond(a,b,c)\defeq a\wedge b\wedge c\,.
\]
Indeed, it is not hard to see that distributivity guarantees that \eqref{MO1}–\eqref{MO4} hold.
To verify \eqref{PI1}, we need to check that the following inequality
\begin{equation}\label{Eq. example}
a\wedge b\wedge f\leq (a\wedge b\wedge \neg d)\vee (a\wedge b\wedge \neg e)\vee (d\wedge e\wedge f)
\end{equation}
holds for every $a,b,d,e,f\in A$. By distributivity and De Morgan's laws,
\begin{align*}
a\wedge b\wedge f
&= (a\wedge b\wedge f\wedge\neg d)
 \vee (a\wedge b\wedge f\wedge\neg e)
 \vee (a\wedge b\wedge f\wedge d\wedge e)\\
&\leq (a\wedge b\wedge\neg d)
 \vee (a\wedge b\wedge\neg e)
 \vee (d\wedge e\wedge f).
\end{align*}
Moreover, since $a\wedge f=\Diamond(a,a,f)$, condition \eqref{PI3} holds.
It is routine to check that \eqref{PI2} and \eqref{PI4} obtain, too.\QED
\end{example}
Furthermore,  if we order the ternary operators on a Boolean algebra in the pointwise manner, the operator defined in Example~\ref{smallest Diamond} is the smallest one that turns a Boolean algebra into a PSI-Algebra.
This fact can be made precise with the following proposition.

\begin{proposition}
Let $(\mathbf{A},\Diamond)$ be any PSI-Algebra. Then, for all $a,b,f\in A$, we have
\[
a\wedge b\wedge f \leq \Diamond(a,b,f).
\]
\end{proposition}

\begin{proof}
By \eqref{PI3}, \eqref{MO3}, and \eqref{MO2}, we obtain
\[
a\wedge b\wedge f
\le \Diamond(a\wedge b, a\wedge b, f)
\le \Diamond(a\wedge b, b, f)
\le \Diamond(a, b, f),
\]
as claimed.
\end{proof}

\begin{remark}
Not every 2-additive 3-BAMO is a PSI-Algebra. To see this, consider the 2-additive 3-BAMO introduced in Example~\ref{example 3BAMO}. Observe that
\[
\Diamond(1,a,1) = 1\quad\tand\quad\Diamond(1,a,a) = a\quad\tand\quad\Diamond(1,a,\neg a)=0=\Diamond(\neg a,a,1)\,.
\]
Hence,
\[
\Diamond(1,a,1) > \Diamond(1,a,a)\vee \Diamond(1,a,\neg a)\vee\Diamond(\neg a,a,1),
\]
which shows that condition~\eqref{PI1} fails.\QED
\end{remark}

\begin{definition}
    A PSI-Algebra $(\mathbf{A}, \Diamond)$ is said to be \emph{relational} if, for every $a,b,c\in A$, \[\Diamond(a,b,c)\in \{0,1\}.\]\QED
\end{definition}

It is clear that the class of 2-additive 3-BAMOs and the class of Pseudo-Inference Algebras are both varieties.

Our aim now is to extract the algebraic features of ECAs. As shown in Lemma~\ref{EC as 3BAMO}, the characteristic function of the extended contact relation~$\vdash$ behaves similarly to the ternary operation~$\Diamond$ of 2-additive 3-BAMOs. Below, we make this connection precise and show how Pseudo-Inference Algebras can serve as algebraic models for Extended Contact Algebras. The guiding idea is to bring to light, by means of the characteristic function of~$\vdash$, the behavior of a suitable operator that captures the essential features of extended contact in purely algebraic terms. We emphasize that this approach extends the ideas developed in~\citep{CJ2022}, \citep{BBSV2016}, and~\citep{BBSV2019} for subordination algebras.

Given an ECA $(\mathbf{A},\vdash)$, we define the ternary operator $\Diamond_{\vdash}\colon A^3\to A$ associated with $\vdash$ by
\begin{equation}\tag{$\mathrm{df}\,\Diamond_{\vdash}$}\label{relation to diamond}
    \Diamond_{\vdash}(a,b,c)\defeq\neg\chi_{\vdash}(a,b,\neg c)\,.
\end{equation}
Also, for a relational Pseudo-Inference Algebra $(\mathbf{A}, \Diamond)$ we define a ternary relation by means of the operator in the following way
\begin{equation}\tag{$\mathrm{df}\,\mathord{\vdash}_{\Diamond}$}\label{diamond to relation}
    (a,b)\vdash_{\Diamond} c \; \iffdef\; \Diamond(a, b,\neg c)=0\,.
\end{equation}

\begin{lemma}\label{extended contact are relational pseudo-iinference algebras}
    The following hold:
    \begin{enumerate}
        \item If $(\mathbf{A}, \vdash)$ is an ECA, then   $(\mathbf{A},  \Diamond_{\vdash})$ is a non-degenerate relational Pseudo-Inference Algebra, and $\vdash_{\Diamond_{\vdash}}=\mathord{\vdash}$.
        \item If $(\mathbf{A}, \Diamond)$ is a non-degenerate relational Pseudo-Inference Algebra, then $(\mathbf{A},  \vdash_{\Diamond})$ is an ECA, and $\Diamond_{\vdash_{\Diamond}}=\Diamond$.
    \end{enumerate}
\end{lemma}
\begin{proof}
(1) Let $(\mathbf{A}, \vdash)$ be an Extended Contact Algebra. From \eqref{relation to diamond} and Lemma \ref{EC as 3BAMO}, it is easily seen that $(\mathbf{A},  \Diamond_{\vdash})$ is a non-degenerate 2-additive 3-BAMO such that $\Diamond_{\vdash}[A^3]=\{0,1\}$. Now we prove conditions \eqref{PI1}--\eqref{PI4}.

\eqref{PI1} Fix arbitrary $a,b,f,d$ and $e$ from $A$. If $\Diamond_{\vdash}(a,b,f)=1$, then, by (\ref{relation to diamond}), we have $(a,b)\nvdash \neg f$. Thus, from \eqref{ExtCA1} we obtain that $(a,b)\nvdash d$ or $(a,b)\nvdash e$ or $(d,e)\nvdash \neg f$. Again by (\ref{relation to diamond}), it must be the case that either $\Diamond_{\vdash}(a,b,\neg d)=1$ or $\Diamond_{\vdash}(a,b,\neg e)=1$ or $\Diamond_{\vdash}(d,e, f)=1$.  From the latter, it is clear that \eqref{PI1} holds. Further, it is also true that if $\Diamond_{\vdash}(a,b,f)=0$, then \eqref{PI1} also holds.

\smallskip

\eqref{PI2} From \eqref{EC2}---assuming that $f= 0$--- and \eqref{relation to diamond} it follows that $\Diamond_{\vdash}(a,b,\neg a)=0$, as required.

\smallskip

\eqref{PI3} Since the set of values of $\Dv$ is $\{0,1\}$ it suffices to consider the case when $\Dv(a,a,f)$ is equal to $0$, that is $(a,a)\vdash\neg f$. By \eqref{EC3} we obtain that $a\leq\neg f$, i.e., $a\wedge f=0$, as required.

\smallskip

 \eqref{PI4} is an immediate consequence of \eqref{ExtCA4}.

 \smallskip

 The equality $\vdash_{\Diamond_{\vdash}}=\mathord{\vdash}$ is a straightforward consequence of the definitions (\ref{relation to diamond}) and (\ref{diamond to relation}).

\smallskip

(2) Let $(\mathbf{A}, \Diamond)$ be a non-degenerate relational Pseudo-Inference Algebra. We show that $\vdash_{\Diamond}$ satisfies
\eqref{ExtCA0}, \eqref{ExtCA1}, \eqref{EC2},
\eqref{EC3} and \eqref{ExtCA4}.

\smallskip

\eqref{ExtCA0} If $(a,b)\vdash_{\Diamond} f$, then (\ref{diamond to relation}) entails that $\Diamond(a,b,\neg f)=0$. Notice that, in order to prove the claim it is enough to show that $\Diamond(a\vee d,b,\neg(f\vee d))=0$. By \eqref{PI2} and Proposition \ref{quasiecuation PI2}, it is the case that $\Diamond(d,b,\neg(f\vee d))=0$. So, since by assumption $\Diamond$ preserves joins in the first coordinate and is monotone in the third one, we have
\begin{align*}
\Diamond(a\vee d,b,\neg(f\vee d))&{}=\Diamond(a,b,\neg(f\vee d))\vee \Diamond(d,b,\neg(f\vee d))\\
&{}=\Diamond(a,b,\neg(f\vee d))\leq \Diamond(a,b,\neg f)=0\,.
\end{align*}

\eqref{ExtCA1} Suppose that \((a, b) \vD d\), \((a, b) \vD e\) and \((d, e) \vD f\). Then, by \eqref{diamond to relation} we have that
\[
\Diamond(a,b,\neg d)=\Diamond(a,b,\neg e)=\Diamond(d,e,\neg f)= 0\,.
\]
Since by \eqref{PI1} it is the case that $\Diamond(a,b,\neg f)\leq \Diamond(a,b,\neg d)\vee \Diamond(a,b,\neg e) \vee  \Diamond(d,e,\neg f)$, we obtain $\Diamond(a,b,\neg f)=0$. Therefore, again by (\ref{diamond to relation}), we conclude that $(a,b)\vD f$, as required.

\smallskip

\eqref{EC2} By \eqref{PI2} and monotonicity we have that $\Diamond(a,b,\neg a\wedge \neg f)=0$, so $(a,b)\vD\neg(\neg a\wedge\neg f)$, and by De Morgan law we obtain $(a,b)\vD a\vee f$.

\smallskip

\eqref{EC3}  Suppose $(a,a)\vD f$. Thus $\Diamond(a,a,\neg f)=0$, and by \eqref{PI3} we obtain that $a\wedge\neg f=0$. It follows that $a\leq f$.

\smallskip

\eqref{ExtCA4} follows from \eqref{PI4} in a straightforward way.

\smallskip

The equality $\Diamond_{\vdash_{\Diamond}}=\Diamond$ is a straightforward consequence of the definitions (\ref{relation to diamond}) and (\ref{diamond to relation}).
\end{proof}

As an immediate consequence of Lemma \ref{extended contact are relational pseudo-iinference algebras} we obtain the following result.

\begin{proposition}\label{relational are ECA}
Non-degenerate relational Pseudo-Inference Algebras and Extended Contact Algebras are in one-to-one correspondence.
\end{proposition}

\begin{corollary}\label{posets ECA RPSI}
    Let $\mathbf{A}$ be a Boolean algebra. The posets  $\mathrm{ECA}(\mathbf{A})$ of Extended Contact Algebras on $\mathbf{A}$ ordered by inclusion, and $\mathrm{RPSI}^{\ast}(\mathbf{A})$ of non-degenerate relational PSI-Algebras on $\mathbf{A}$ ordered pointwise, are dually isomorphic.
\end{corollary}
\begin{proof}
    Straightforward from Proposition~\ref{relational are ECA}, \eqref{relation to diamond} and \eqref{diamond to relation}.
\end{proof}

\begin{remark}
Every Boolean algebra \( \mathbf{A} \) can be naturally turned into a relational PSI-Algebra. Indeed, if we define
\[
(a,b) \vdash c \iffdef a \wedge b \wedge \neg c = 0,
\]
then it is immediate that \( (\mathbf{A}, \vdash) \) is an ECA. Moreover, by \eqref{ExtCA3}, the relation \( \vdash \) is the largest extended contact relation on \( \mathbf{A} \). By Lemma \ref{extended contact are relational pseudo-iinference algebras}, \( (\mathbf{A}, \Diamond_{\vdash}) \) is a relational PSI-Algebra, and by Corollary~\ref{posets ECA RPSI}, it is the smallest relational PSI-Algebra on \( \mathbf{A} \).\QED
\end{remark}

\section{The variety generated by relational PSI-Algebras}\label{Section 4}
In this section, we show that the class of relational PSI-Algebras generates the subvariety of strict PSI-Algebras and that this subvariety is a discriminator variety.

\begin{definition}\label{Strict psi-algebras}
A PSI-Algebra $(\mathbf{A},\Diamond)$ is said to be \emph{strict} if it satisfies the following additional equations:
\begin{align}
\label{R1}\tag{R1}  \Diamond(x,y,a)\wedge \neg\Diamond(x,y,b)&{}\leq \Diamond(1,1,a\wedge \neg b),\\
\label{R2}\tag{R2}  \Diamond(x,a,y)\wedge \neg\Diamond(x,b,y)&{}\leq \Diamond(1,a\wedge \neg b,1),\\
\label{S}\tag{S}  \Diamond(a,b,c)&{}\leq \mu(\Diamond(a,b,c)),
\end{align}
where $\mu(z)\defeq \neg \Diamond(1,1,\neg z)\wedge \neg \Diamond(1,\neg z,1)$. We write $\mathcal{SPSI}$ for the variety of strict PSI-Algebras.\QED
\end{definition}

\begin{lemma}\label{RelPsi R1,R2,I,S}
    In every relational PSI-Algebra,  \eqref{R1}, \eqref{R2} and \eqref{S} hold.
\end{lemma}
\begin{proof}
Since $\mu(1)=1$ and $\Diamond(a,b,c)\in\{0,1\}$ in every
relational PSI-Algebra, condition \eqref{S} holds.

\smallskip

\eqref{R1} Suppose that $\Diamond(x,y,a)\wedge \neg\Diamond(x,y,b)=1$. Then, $\Diamond(x,y,a)=1$ and $\Diamond(x,y,b)=0$. It follows that $a\wedge \neg b>0$. Indeed, if $a\wedge \neg b=0$, then $a\leq b$, and by \eqref{MO4} we obtain $1=\Diamond(x,y,a)\leq \Diamond(x,y,b)=0$, a contradiction. By \eqref{PI3}, we have $0<a\wedge \neg b\leq \Diamond(1,1,a\wedge \neg b)$, which implies that $\Diamond(1,1,a\wedge \neg b)=1$, as required.

\smallskip

\eqref{R2} Suppose that $\Diamond(x,a,y)\wedge \neg\Diamond(x,b,y)=1$. Then, $\Diamond(x,a,y)=1$ and $\Diamond(x,b,y)=0$. By the same reasoning as before, it follows that $a\wedge \neg b>0$. By \eqref{PI3} and \eqref{MO2}, we have $0<a\wedge \neg b\leq \Diamond(a\wedge \neg b,a\wedge \neg b,1)\leq  \Diamond(1,a\wedge \neg b,1)$, which leads to $\Diamond(1,a\wedge \neg b,1)=1$.

This concludes the proof.
\end{proof}

Let $(\mathbf{A},\Diamond)$ be a PSI-Algebra. We write $a\to b\defeq \neg a\vee b$ for the Boolean implication.
A filter $F$ of $\mathbf{A}$ is said to be \emph{$\Diamond$-closed} if
 \[
\Diamond(x_1,x_2,x_3)\to \Diamond(y_1,y_2,y_3)\in F
\]
whenever $x_i\to y_i\in F$ for every $1\leq i\leq 3$.

In \citep[Theorem~27 and Proposition~29]{Celani2009}, the correspondence between congruences
and closed filters was established for general BAMOs.
Applying this result to 2-additive 3-BAMOs, we obtain that the congruences
of PSI-Algebras are in one-to-one correspondence with their
$\Diamond$-closed filters.

\begin{lemma}\label{Reduction Closed Filters}
     Let $(\mathbf{A}, \Diamond)$ be a PSI-Algebra, and let $F$ be a filter of $\mathbf{A}$. Then, $F$ is $\Diamond$-closed if and only if, for every $a,b,x,y\in A$ the following conditions hold:
     \begin{align}
\label{(Su)}\tag{Su} \text{If}\; a\to b \in F \;\text{then}\; &\Diamond(x,y,a)\to\Diamond(x,y,b)\in F\,.\\
\label{(Mid)}\tag{Mid} \text{If}\; a\to b \in F \;\text{then}\; &\Diamond(x,a,y)\to\Diamond(x,b,y)\in F\,.
\end{align}
\end{lemma}
\begin{proof}
 Let us assume that $F$ is $\Diamond$-closed. If $a\rightarrow b\in F$, then, since $x\rightarrow x=y\rightarrow y=1\in F$, both $\eqref{(Su)}$ and \eqref{(Mid)} follow immediately. For the converse, let us assume that $F$ satisfies both $\eqref{(Su)}$ and $\eqref{(Mid)}$ and suppose that for all $1\leq i\leq 3$, $a_i\rightarrow b_i\in F$. By \eqref{(Mid)} $\Diamond(a_2,a_1,a_3)\to \Diamond(a_2,b_1,a_3)\in F$ and (a) $\Diamond(b_1,a_2,a_3)\to \Diamond(b_1,b_2,a_3)\in F$. Observe that \eqref{PI4} implies that $\Diamond(a_2,a_1,a_3)=\Diamond(a_1,a_2,a_3)$ and $\Diamond(a_2,b_1,a_3)=\Diamond(b_1,a_2,a_3)$, so we get (b) $\Diamond(a_1,a_2,a_3)\rightarrow \Diamond(b_1,a_2,a_3)\in F$. Furthermore, by \eqref{(Su)}, we obtain (c) $\Diamond(b_1,b_2,a_3)\rightarrow \Diamond(b_1,b_2,b_3)\in F$. Thus, from (a) and (b), we get that
 \[
 [\Diamond(a_1,a_2,a_3)\to \Diamond(b_1,a_2,a_3)]\wedge [\Diamond(b_1,a_2,a_3)\rightarrow \Diamond(b_1,b_2,a_3)]\in F.
 \]
 Since the identity $(p\rightarrow q)\wedge (q\rightarrow r)\leq p\rightarrow r$ holds in every Boolean algebra, it is the case that $\Diamond(a_1,a_2,a_3)\to \Diamond(b_1,b_2,a_3)\in F$. Thus, by (c)
 \[
 [\Diamond(a_1,a_2,a_3)\to \Diamond(b_1,b_2,a_3)]\wedge [\Diamond(b_1,b_2,a_3)\rightarrow \Diamond(b_1,b_2,b_3)] \in F,
 \]
 and consequently $\Diamond(a_1,a_2,a_3)\to \Diamond(b_1,b_2,b_3)\in F$, which proves that $F$ is $\Diamond$-closed, as desired.
\end{proof}

For the following result, we define $\mu^{0}(x)\defeq x$ and $\mu^{l+1}(x)\defeq\mu(\mu^l(x))$, for every $l\geq 0$.

\begin{proposition}\label{Properties of mu}
    Let $(\mathbf{A}, \Diamond)$ be a PSI-Algebra. The following hold:
    \begin{enumerate}
        \item $\mu(0)=0$.
        \item $ \mu(x)\leq x$ for all $x\in A$.
        \item $\mu(z)=1$ if and only if $z=1$.
        \item $\mu$ is a monotone operator.
        \item $\mu^{n+1}(x)\leq \mu^{n}(x)$, for every $n\geq 0$.
    \end{enumerate}
\end{proposition}
\begin{proof}
(1) From  \eqref{PI3}, $1=\Diamond(1,1,1)$, thus $\mu(0)= \neg \Diamond(1,1,1)=0$.

\smallskip

(2) From \eqref{PI3}, $\neg x\leq \Diamond(1,1,\neg x)$, so $\mu(x)\leq \neg \Diamond(1,1,\neg x)\leq x$.

\smallskip

(3) By definition of $\mu$, we obtain $\mu(1)= \neg \Diamond(1,1,0)\wedge \neg \Diamond(1,0,1)$, so from \eqref{MO1}, it follows that $\mu(1)=1$. On the other hand, suppose that $\mu(z)=1$. Then, by (2), $1= \mu(z)\leq z $ and it follows that $z=1$.

\smallskip

(4) If $x\leq y$, then $\neg y\leq \neg x$. From the monotonicity of $\Diamond$ in each coordinate, it follows that $\neg \Diamond(1,\neg x,1)\leq \neg \Diamond(1,\neg y,1)$ and $\neg \Diamond(1,1,\neg x)\leq \neg \Diamond(1,1,\neg y)$. Therefore, we may conclude that $\mu(x)\leq \mu(y)$.

\smallskip

(5) We proceed by induction on $l$. From (2), $\mu(x)\leq x=\mu^0(x)$.
Suppose that $\mu^{l+1}(x)\leq \mu^{l}(x)$. Then, from (4), $\mu^{l+2}(x)=\mu(\mu^{l+1}(x))\leq \mu(\mu^{l}(x))=\mu^{l+1}(x)$. Hence (5) holds, as claimed.
\end{proof}

\begin{definition}\label{def Modal filter}
    Let $(\mathbf{A}, \Diamond)$ be a PSI-Algebra. A filter $F$ of $\mathbf{A}$ is said to be \emph{modal} if it is closed under $\mu$.\QED
\end{definition}

\begin{lemma}\label{Closed modal}
    Let $(\mathbf{A}, \Diamond)$ be a PSI-Algebra. Then, every $\Diamond$-closed filter of $\mathbf{A}$ is modal.
\end{lemma}
\begin{proof}
    Let $F$ be a $\Diamond$-closed filter such that $a\in F$. Then $\neg a\rightarrow 0\in F$. By Lemma \ref{Reduction Closed Filters}, $F$ satisfies both \eqref{(Su)} and \eqref{(Mid)}, so $\Diamond(1,1,\neg a)\rightarrow \Diamond(1,1,0)\in F$ and $\Diamond(1,\neg a,1)\rightarrow \Diamond(1,0,1)\in F$. Therefore, by \eqref{MO1}, $\neg\Diamond(1,1,\neg a)\in F$ and $\neg\Diamond(1,\neg a,1)\in F$. Since $F$ is a filter, the latter yields that $\mu(a)\in F$, as desired.
\end{proof}

\begin{lemma}\label{Modal R1R2 closed}
Let $(\mathbf{A}, \Diamond)$ be a PSI-Algebra satisfying \eqref{R1} and \eqref{R2}. Then, every modal filter is $\Diamond$-closed.
\end{lemma}
\begin{proof}
    Let $F$ be a modal filter of $\mathbf{A}$. By Lemma \ref{Reduction Closed Filters}, in order to prove our claim it is enough to show that $F$ satisfies both \eqref{(Su)} and \eqref{(Mid)}. To do so, suppose that $a\rightarrow b\in F$. Since $F$ is modal, we have that $\mu(a\rightarrow b)\in F$. From \eqref{R1} it follows that
    \[
    \Diamond(x,y,a)\wedge \neg\Diamond(x,y,b)\leq \Diamond(1,1,a\wedge \neg b)
    \]
    and thus
    \[
    \mu(a\rightarrow b)\leq \neg\Diamond(1,1,\neg (a\rightarrow b))=\neg\Diamond(1,1,a\wedge \neg b)\leq  \Diamond(x,y,a) \rightarrow \Diamond(x,y,b).
    \]
    In a similar fashion, from \eqref{R2}, we obtain
    \[
    \mu(a\rightarrow b)\leq \neg\Diamond(1,\neg (a\rightarrow b),1)=\neg\Diamond(1,a\wedge \neg b,1)\leq  \Diamond(x,a,y) \rightarrow \Diamond(x,b,y).
    \]
    Since $\mu(a\rightarrow b)\in F$ and $F$ is upward closed, both \eqref{(Su)} and \eqref{(Mid)} hold for $F$, as required.
\end{proof}

If $\mathbf{A}$ is a Boolean algebra and  $x\in A$, then the set
\[
[x)\defeq\{a\in A\colon x\leq a\}
\]
is the smallest filter of $\mathbf{A}$ containing $x$.

An algebra $\mathbf{A}$ is called \emph{simple} if its only congruences
are the identity relation $\Delta_A=\{(a,a):a\in A\}$ and the universal
relation $A^2$, see \citep[Definition II.8.8]{BS2011}.

\begin{theorem}\label{relational psi are simple}
  A PSI-Algebra $(\mathbf{A}, \Diamond)$ is relational if and only if it is simple and satisfies \eqref{R1}, \eqref{R2}, and \eqref{S}.
\end{theorem}
\begin{proof}
    If $(\mathbf{A}, \Diamond)$ is relational, then the equations hold by Lemma \ref{RelPsi R1,R2,I,S}. It remains to show that $(\mathbf{A}, \Diamond)$ is simple. In view of Lemmas \ref{Closed modal} and \ref{Modal R1R2 closed} and the correspondence between congruences and $\Diamond$-closed filters, it suffices to prove that the only modal filters of $\mathbf{A}$ are $\{1\}$ and $A$. Suppose that a modal filter $F$ is such that $F\neq \{1\}$, i.e., there exists $a\in F$ such that $a\neq 1$. By the assumption, $\mu(a)\in \{0,1\}$ and by Proposition \ref{Properties of mu}(3), it follows that $\mu(a)\neq 1$. Therefore, $\mu(a)=0\in F$ and we obtain $F=A$. Hence, $(\mathbf{A}, \Diamond)$ is simple, as claimed.

   For the converse implication, let $(\mathbf{A},\Diamond)$ be a simple PSI-Algebra satisfying
\eqref{R1}, \eqref{R2}, and \eqref{S}.
To show that $(\mathbf{A},\Diamond)$ is relational,
pick $a,b,c\in A$ and consider the element $\Diamond(a,b,c)$. By the simplicity of $(\mathbf{A},\Diamond)$,
the only modal filters of $\mathbf{A}$ are $\{1\}$ and~$A$.
Moreover, from \eqref{S}
and Proposition~\ref{Properties of mu}, it follows that $\Diamond(a,b,c)=\mu(\Diamond(a,b,c))$ and by monotonicity of $\mu$, we obtain that $[\Diamond(a,b,c))$ is a modal filter. So, it is either $\{1\}$ or $A$.
In the first case,
$\Diamond(a,b,c)=1$. In the second case, $[\Diamond(a,b,c))=A$, which
occurs exactly when $\Diamond(a,b,c)$ is the least element $0$. Therefore, for all $a,b,c\in A$, we have
$\Diamond(a,b,c)\in\{0,1\}$. Consequently,  $(\mathbf{A},\Diamond)$ is relational.
\end{proof}

\begin{lemma}\label{S implies good stuff}
    Let $(\mathbf{A}, \Diamond)$ be a PSI-Algebra satisfying \eqref{S}. Then, for all $a\in A$
    \[\neg \mu(a)=\mu(\neg \mu(a)).\]
    Moreover, $\mu^{l}(\neg \mu(a))=\neg \mu(a)$ for every $l\geq 0$.
\end{lemma}
\begin{proof}
    It is immediate from the definition of $\mu$, that $\Diamond(1,\neg a,1), \Diamond(1,1,\neg a)\leq \neg \mu(a)$. So, Proposition \ref{Properties of mu} and \eqref{S} yield
    \[\neg \mu(a)= \Diamond(1,\neg a,1)\vee \Diamond(1,1,\neg a)\leq \mu(\Diamond(1,\neg a,1))\vee  \mu(\Diamond(1,1,\neg a))\leq \mu(\neg \mu(a))\leq \neg \mu(a)\]
as claimed. For the remaining part, observe that
\[
\mu^2(\neg \mu(a))=\mu(\mu(\neg \mu(a)))=\mu(\neg \mu(a))=\neg \mu(a).
\]
The result follows by a straightforward induction on~$l$.
\end{proof}

We recall that in every Boolean algebra $\mathbf{A}$ there is a one-to-one correspondence between congruences and filters. Explicitly, if $\theta$ is a congruence of $\mathbf{A}$, then
\[
F_{\theta}\defeq \{\, a \in A : (a,1) \in \theta \,\}
\]
is a filter of $\mathbf{A}$; conversely, if $F$ is a filter of $\mathbf{A}$, then
\[
\theta_{F}\defeq \{\, (a,b) \in A^{2} : a \wedge f = b \wedge f \text{ for some } f \in F \,\}
\]
is a congruence such that
\[
\theta_{F_{\theta}} = \theta \quad \text{and} \quad F = F_{\theta_{F}}.
\]

A nontrivial algebra is \emph{subdirectly irreducible} if it has a
least congruence strictly larger than the identity congruence.
A \emph{subdirect product} is a subalgebra of a direct product whose
coordinate projections are all surjective. An algebra is
\emph{semisimple} if it is isomorphic to a subdirect product of simple
algebras, and a variety is \emph{semisimple} if all its members are semisimple
\citep[Theorem II.8.4 and Definition IV.12.1]{BS2011}.

\begin{theorem}\label{SPSI is semisimple}
    The variety $\mathcal{SPSI}$ is semisimple, its simple members are precisely the relational PSI-Algebras, and it is generated by the class of relational PSI-Algebras.
\end{theorem}
\begin{proof}
    By Theorem~\ref{relational psi are simple}, the simple members of $\mathcal{SPSI}$ are precisely the relational PSI-Algebras.
    To prove semisimplicity, we show that every subdirectly irreducible member of $\mathcal{SPSI}$ is simple.  Suppose that $(\mathbf{A}, \Diamond)\in\mathcal{SPSI}$ is subdirectly irreducible. Then there exists the smallest non-identity congruence~$\theta$. By Lemmas~\ref{Closed modal} and \ref{Modal R1R2 closed}, the congruence--filter correspondence implies that $F_\theta$ is the least nontrivial modal filter of~$\mathbf{A}$ (i.e., minimal among the modal filters different from $\{1\}$). Thus, there exists $a\in F_\theta$ such that $a\neq 1$. From Proposition~\ref{Properties of mu}(3), it follows that $\mu(a)\neq 1$, and therefore $\neg \mu(a)\neq 0$. By Lemma~\ref{S implies good stuff} and the monotonicity of~$\mu$, we have that $[\neg \mu(a))$ is a proper modal filter. If $[\neg\mu(a))\neq\{1\}$, then $F_\theta\subseteq[\neg\mu(a))$; otherwise, $[\neg\mu(a))=\{1\}$. In the first case, since $F_\theta$ is modal and $a\in F_\theta$, it follows that $\mu(a)\in F_\theta\subseteq [\neg \mu(a))$. Then $\neg \mu(a)\leq \mu(a)$, and we obtain $\mu(a)=1$, a contradiction in light of Proposition~\ref{Properties of mu}(3), since $a\neq 1$. Therefore, $[\neg \mu(a))=\{1\}$, which implies $\mu(a)=0$. Again, since $F_\theta$ is modal and $a\in F_\theta$, it follows that $0\in F_\theta$, and thus $F_\theta=A$. Hence, $(\mathbf{A}, \Diamond)$ is simple.

    Finally, by Birkhoff's subdirect representation theorem \citep[Theorem II.8.6]{BS2011}, every member of $\mathcal{SPSI}$ is a subdirect product of subdirectly irreducible members. Since every subdirectly irreducible member of $\mathcal{SPSI}$ is simple, and thus a relational PSI-Algebra, it follows that $\mathcal{SPSI}$ is generated by the class of relational PSI-Algebras. This concludes the proof.
\end{proof}

Let $\mathcal{K}$ be a class of algebras. A ternary term $t(x, y, z)$ is called a \emph{discriminator term} for the class $\mathcal{K}$ if, in every algebra $\mathbf{A} \in \mathcal{K}$, it satisfies
\[
t(a, b, c) =
\begin{cases}
a, & \text{if } a \ne b,\\[4pt]
c, & \text{if } a = b,
\end{cases}
\]
for all $ a, b, c \in A$. The variety generated by $\mathcal{K}$ is then called a \emph{discriminator variety}.

Suppose that $\mathcal{K}$ is contained in a variety
of Boolean algebras expanded by additional operations in a fixed
signature. Then the existence of a discriminator term for
$\mathcal{K}$ is equivalent to the existence of a unary term \(d(x) \)---called a \emph{unary discriminator term}---which satisfies, in every member of \(\mathcal{K} \), the following conditions:
\[
d(a)=
\begin{cases}
    0, & \text{if } a=0,\\[4pt]
    1, & \text{if } a\neq 0.
\end{cases}
\]
It is worth noticing that if we consider the symmetric difference
\[
x + y\defeq (x\wedge \neg y) \vee (\neg x\wedge y),
\]
then, in such a case, the terms \(d \) and \(t \) are in fact interdefinable,
\[
d(x) = \neg t(0, x, 1)
\quad \text{and} \quad
t(x, y, z) = (x\wedge d(x + y)) \vee (z\wedge \neg d(x + y)).
\]

\begin{theorem}\label{SPSI is a discriminator}
    The variety $\mathcal{SPSI}$ is a discriminator variety.
\end{theorem}
\begin{proof}
    By Theorem \ref{SPSI is semisimple}, to prove our claim it is enough to show that the class of relational PSI-Algebras has a discriminator term. To do so, we define $d(x)\defeq\neg \mu (\neg x)$ and show that $d$ is a unary discriminator term. Let $(\mathbf{A},\Diamond)$ be a relational PSI-Algebra. By Proposition~\ref{Properties of mu}(3), $d(0)=0$. So let us assume that $a\neq 0$. If $d(a)\neq 1$, it must be that $d(a)=0$, so $\Diamond(1,a,1)=\Diamond(1,1,a)=0$. By \eqref{PI3}, we get $a=a\wedge 1\leq \Diamond(1,1,a)$, which is a contradiction. Therefore, $d(a)=1$ and consequently $d$ is a unary discriminator term for the class of relational PSI-Algebras, as required.
\end{proof}

A variety \(\mathcal{V} \) is said to have the \emph{congruence extension property} (abbreviated as CEP) if, for every subalgebra \(\mathbf{A} \) of an algebra \(\mathbf{B} \in \mathcal{V} \) and for each congruence \(\delta \) on \(\mathbf{A} \), there exists a congruence \(\theta \) on \(\mathbf{B} \) such that \(\delta = \theta \cap A^{2} \). Moreover, \(\mathcal{V} \) is said to have \emph{definable principal congruences} (DPC) if there exists a formula \(\zeta(x, y, u, v) \) in the first-order language of \(\mathcal{V} \) such that, for every algebra \(\mathbf{A} \in \mathcal{V} \) and all \(a, b, c, d \in A \),
\[
(c, d) \in \mathsf{Cg}^{\mathbf{A}}(a, b)
\quad \Longleftrightarrow \quad
\mathbf{A} \models \zeta[a, b, c, d],
\]
where $\mathsf{Cg}^{\mathbf{A}}(a, b)$ denotes the smallest congruence on $\mathbf{A}$ containing $(a,b)$. If the formula \(\zeta(x, y, u, v) \) can be expressed as a finite conjunction of equations, then \(\mathcal{V} \) is said to have \emph{equationally definable principal congruences} (EDPC).

We conclude this section with two consequences of the discriminator property.

\begin{corollary}
    The following hold:
    \begin{enumerate}
        \item $\mathcal{SPSI}$ has CEP.
        \item $\mathcal{SPSI}$ has EDPC.
    \end{enumerate}
\end{corollary}
\begin{proof}
By Theorem~\ref{SPSI is a discriminator}, $\mathcal{SPSI}$ is a discriminator variety, and every discriminator variety has CEP and EDPC.
\end{proof}

\section{Descriptive frames and representation}\label{Section 5}

In this section, we introduce descriptive PSI-frames and establish
the representation of 2-additive 3-BAMOs, building on the framework
of \citep{Celani2009}. This provides the basis for the study of
PSI-spaces in Section~\ref{sec:psi-spaces} and the dual equivalences
in Section~\ref{sec:morphisms-dualities}.

We begin by fixing the notation that will be used throughout the rest of the paper.
Recall that a \emph{Boolean space} is a compact Hausdorff space with a basis of clopen sets. Let $(X,\tau)$ be a Boolean space.

We denote by $\mathscr{C}(X)$ the collection of all closed subsets of $(X, \tau)$, and by $\mathscr{C}_0(X)$ the family $\mathscr{C}(X)\setminus \{\emptyset\}$.
For $\vec{Y} = (Y_1, Y_2, Y_3)$ and $\vec{Z} = (Z_1, Z_2, Z_3)$ in $\mathscr{C}(X)^3$, we write $\vec{Y} \subseteq \vec{Z}$ to indicate that $\vec{Y}$ is below $\vec{Z}$ with respect to the product order on $\mathscr{C}(X)^3$, that is, $Y_i \subseteq Z_i$ for all $1 \leqslant i \leqslant 3$.
Moreover, we write $\vec{Y} \cap \vec{Z} \neq \varnothing$ whenever there exists some $1 \leqslant i \leqslant 3$ such that $Y_i \cap Z_i \neq \varnothing$. Thus, $\vec{Y}\cap\vec{Z}=\emptyset$ means that all three componentwise intersections are empty.
For any subset $Y\subseteq X$, we define $Y^c\defeq X\setminus Y$. For a triple $\vec{Y}= (Y_1, Y_2, Y_3)$, $\vec{Y}^c\defeq (Y_1^c, Y_2^c, Y_3^c)$ is the componentwise complement of $\vec{Y}$ (not its complement as a subset of $X^3$). The components of $\vec{Y}^c$ need not be closed.

Let $\CO(X)$ be the family of all clopen subsets of $(X, \tau)$. Thus, $\CO(X)^3$ is the set of ordered triples of clopen subsets of $X$. We define $\vec{\emptyset}\defeq(\emptyset,\emptyset,\emptyset)$ and $\vec{X}\defeq(X,X,X)$.

For $\vec{U} \in \CO(X)^3$ let
\[
L_{\vec{U}}\defeq\left\{ \vec{Y} \in \mathscr{C}_0(X)^3: \vec{Y} \cap \vec{U} \neq \emptyset \right\}.
\]
Complements of families of triples, such as $L_{\vec{U}}^c$, are taken in $\mathscr{C}_0(X)^3$. In particular,
\[
L_{\vec{U}^c}^{c}
=\{\vec{Y}\in\mathscr{C}_0(X)^3:\vec{Y}\subseteq \vec{U}\}.
\]

\begin{remark}
Let $\vec{U}, \vec{V} \in \CO(X)^3$. For
\[
\vec{U} \cup \vec{V} \defeq (U_1 \cup V_1,\, U_2 \cup V_2,\, U_3 \cup V_3), \quad\text{and}\quad
\vec{U} \cap \vec{V}\defeq (U_1 \cap V_1,\, U_2 \cap V_2,\, U_3 \cap V_3)
\]
the following properties hold:
\begin{enumerate}
    \item $L_{\vec{\emptyset}} = \emptyset$,
    \item $L_{\vec{X}} = \mathscr{C}_0(X)^3$,
    \item $L_{\vec{U} \cup \vec{V}} = L_{\vec{U}} \cup L_{\vec{V}}$,
    \item $L_{\vec{U} \cap \vec{V}} \subseteq L_{\vec{U}} \cap L_{\vec{V}}$.\QED
\end{enumerate}
\end{remark}

For $R \subseteq X \times \mathscr{C}_0(X)^3$, $x\in X$ and $\vec{Y}\in\mathscr{C}(X)^3$, put
\[
R(x)\defeq\{\vec{Z}\in\mathscr{C}_0(X)^3:(x,\vec{Z})\in R\},
\qquad R^{-1}(\vec{Y})\defeq\{x\in X:(x,\vec{Y})\in R\}.
\]
Following the relational construction in \citep{Celani2009}, for $\vec{U}\in\CO(X)^3$ we define
\[
\Diamond_R(\vec{U})\defeq \{x \in X \;\colon\; R(x) \cap L_{\vec{U}^c}^c \neq \emptyset \}, \qquad
\Box_R(\vec{U})\defeq \{x \in X \;\colon\; R(x) \subseteq L_{\vec{U}} \}.
\]

The conditions \eqref{DF1} and \eqref{DF2} below are the ternary instances of the descriptiveness conditions in \citep[Definition 12]{Celani2009}. We impose the additional condition \eqref{DF3} to reflect additivity in the first two coordinates.

\begin{definition}\label{descriptive psi-frame}
    A \emph{descriptive PSI-frame} (a \emph{PSI-frame}, for short) is a structure $(X, \tau, R)$ such that $(X, \tau)$ is a Boolean space and $R \subseteq X \times \mathscr{C}_0(X)^3$ satisfies the following conditions:
\begin{enumerate}[label=(DF\arabic*), ref=DF\arabic*]
    \item\label{DF1} For every $\vec{U} \in \CO(X)^3$, we have
    $\Diamond_R(\vec{U}) \in \CO(X)$.
    \item\label{DF2} For each $x \in X$, $R(x) = \bigcap \left\{ L_{\vec{U}} \;\colon\; x \in \Box_R(\vec{U})\ \text{and}\
        \vec{U} \in \CO(X)^3 \right\}
    $.
    \item\label{DF3} For every $x\in X$ and for every $\vec{Y}\in R(x)$, there exist $y_1\in Y_1$ and $y_2\in Y_2$ such that $(\{y_1\},\{y_2\},Y_3)\in R(x)$.
\end{enumerate}
\end{definition}

\begin{proposition}\label{psi-frames are m-frames}
    Let $(X,\tau,R)$ be a PSI-frame. Then, for all $x\in X$, $R(x)$ is upward closed with respect to the pointwise order on $\mathscr{C}_0(X)^3$: if $\vec{Y}\in R(x)$ and $\vec{Y}\subseteq\vec{Z}\in\mathscr{C}_0(X)^3$, then $\vec{Z}\in R(x)$ (for this reason, the relation is called monotone). Moreover,
    \[
    \Diamond_R(\vec{U})=R^{-1}(\vec{U})\,,
    \]
    for every $\vec{U}\in \CO(X)^3$.
\end{proposition}
\begin{proof}
Suppose that $\vec{Y} \subseteq \vec{Z}$ in $\mathscr{C}_0(X)^3$, and assume that $\vec{Y} \in R(x)$. We proceed by contradiction assuming that  $\vec{Z} \notin R(x)$. Then, by \eqref{DF2}, there exists $\vec{V} \in \CO(X)^3$ such that $R(x) \subseteq L_{\vec{V}}$ and $\vec{Z} \notin L_{\vec{V}}$. By the definition of $L_{\vec{V}}$, this implies $\vec{Z} \cap \vec{V} = \emptyset$. Since $\vec{Y} \subseteq \vec{Z}$, we also have $\vec{Y} \cap \vec{V} = \emptyset$, which means that $\vec{Y} \notin L_{\vec{V}}$. This contradicts the assumption $R(x) \subseteq L_{\vec{V}}$, because $\vec{Y} \in R(x)$. Therefore, $\vec{Z} \in R(x)$, as required.

For the second part, observe that the monotonicity of $R(x)$ with respect to componentwise inclusion of triples of closed sets implies that
\[
\Diamond_R(\vec{U}) = \{ x \in X \mid \text{there exists } \vec{Y} \in R(x) \text{ with } \vec{Y} \subseteq \vec{U} \} = R^{-1}(\vec{U}).\qedhere
\]
\end{proof}

\begin{corollary}\label{psi-frame to 3BAMO}
    If $(X,\tau, R)$ is a PSI-frame, then $(\CO(X),\Diamond_{R})$ is a 2-additive 3-BAMO.
\end{corollary}
\begin{proof}
    It is readily seen that $(\CO(X),\Diamond_{R})$ satisfies \eqref{MO1}. To verify that \eqref{MO2} and \eqref{MO3} hold, observe that Proposition \ref{psi-frames are m-frames} reduces the task to showing that, for all clopen sets $U_1, U_2, U_3, V $ of $(X,\tau)$, the following hold:
\begin{enumerate}
    \item \(R^{-1}(U_1\cup V,\, U_2,\, U_3)\subseteq R^{-1}(U_1,\, U_2,\, U_3)\,\cup\, R^{-1}(V,\, U_2,\, U_3)\),
    \item \(R^{-1}(U_1,\, U_2\cup V,\, U_3)\subseteq R^{-1}(U_1,\, U_2,\, U_3)\,\cup\, R^{-1}(U_1,\, V,\, U_3)\).
\end{enumerate}
The reverse inclusions, as well as \eqref{MO4}, are consequences of the above-mentioned proposition.

We prove (1); the argument for (2) is analogous.
Let \(x\in R^{-1}(U_1\cup V,\, U_2,\, U_3)\). By \eqref{DF3}, there exist \(y\in U_1\cup V\) and \(y_2\in U_2\) such that \((\{y\},\{y_2\},U_3)\in R(x)\). Then:
\begin{itemize}
    \item[(a)] either \(y\in U_1\), in which case \(\{y\}\subseteq U_1\) and \(\{y_2\}\subseteq U_2\). Hence, by Proposition \ref{psi-frames are m-frames}, we obtain \((U_1,U_2,U_3)\in R(x)\), so \(x\in R^{-1}(U_1,U_2,U_3)\subseteq R^{-1}(U_1,U_2,U_3)\cup R^{-1}(V,U_2,U_3)\).
    \item[(b)] or \(y\in V\), and the same reasoning yields \(x\in R^{-1}(V,U_2,U_3)\subseteq R^{-1}(U_1,U_2,U_3)\cup R^{-1}(V,U_2,U_3)\).
\end{itemize}

Thus (1) follows, and the proof is complete.
\end{proof}

Let $\mathbf{A}$ be a Boolean algebra. We write $\mathrm{Fi}(\mathbf{A})$ for its lattice of filters, including the improper filter $A$, and $\Uf(\mathbf{A})$ for its set of ultrafilters, that is, maximal proper filters. A~filter is proper if it does not contain $0$; it is an ultrafilter precisely when it is proper and contains either $a$ or $\neg a$ for every $a\in A$.
The \emph{Stone topology} $\tau_{\mathbf{A}}$ on $\Uf(\mathbf{A})$ has as a basis the family
\[\{\beta(a)\colon a\in A\},\]
where
\[\beta(a)\defeq\{U\in \Uf(\mathbf{A}): a\in U\}.\]
By the Stone representation theorem, $(\Uf(\mathbf{A}),\tau_{\mathbf{A}})$ is a Boolean space and $\beta$ is a Boolean isomorphism from $\mathbf{A}$ onto the algebra of clopen sets of the space of ultrafilters. In particular,
\[
\begin{gathered}
\beta(0)=\emptyset,\qquad \beta(1)=\Uf(\mathbf{A}),\qquad
\beta(\neg a)=\beta(a)^c,\\
\beta(a\wedge b)=\beta(a)\cap\beta(b),\qquad
\beta(a\vee b)=\beta(a)\cup\beta(b),\\
a\leq b\quad\Longleftrightarrow\quad\beta(a)\subseteq\beta(b).
\end{gathered}
\]
Every clopen set is therefore of the form $\beta(a)$.

The lattice of filters of $\mathbf{A}$ and the lattice of closed subsets of its Stone space are dually isomorphic.
A dual isomorphism is established by the following assignments:
\begin{displaymath}
\begin{array}{cc}
  \xymatrix{
    \mathrm{Fi}(\mathbf{A}) \ar[r]^-{\varphi}& \mathscr{C}(\Uf(\mathbf{A})) \\
    F \ar@{|->}[r] & \{U\in\Uf(\mathbf{A})\colon F\subseteq U\}
    }
    &
    \xymatrix{
     \mathscr{C}(\Uf(\mathbf{A})) \ar[r]^-{\psi} & \mathrm{Fi}(\mathbf{A}) \\
    Y \ar@{|->}[r] &  \{a\in A\colon Y\subseteq \beta(a)\}.
    } \\
\end{array}
\end{displaymath}
In what follows, we write $Y_F$ and $F_Y$ to denote $\varphi(F)$ and $\psi(Y)$, respectively.
Thus, for any filter $F$ and any closed set $Y$, we have
\[
F_{Y_F}=F,\qquad Y_{F_Y}=Y=\bigcap_{a\in F_Y}\beta(a).
\]
The first identity follows from the fact that a filter is the intersection of all the ultrafilters that contain it; the second follows from the fact that every closed subset of a Boolean space is the intersection of the clopen sets that contain it. Moreover, if $a\in A$ and $U\in\Uf(\mathbf{A})$, then
\[
F_{\beta(a)}=[a),\qquad F_{\{U\}}=U.
\]
Indeed, $\beta(a)\subseteq\beta(b)$ is equivalent to $a\leq b$, whereas $\{U\}\subseteq\beta(b)$ is equivalent to $b\in U$. Also, $Y\neq\emptyset$ exactly when $F_Y$ is proper.

\smallskip

For a 2-additive 3-BAMO $\mathfrak{A}\defeq(\mathbf{A},\Diamond)$, we use the canonical relation from \citep{Celani2009}, denoted there by $R_{\mathbf{A}}$. In our notation it is the relation $R_{\Diamond}\subseteq \Uf(\mathbf{A})\times \closedne(\Uf(\mathbf{A}))^3$ defined by
\begin{equation}\tag{$\mathrm{df}\,R_{\Diamond}$}\label{def: R diamond}
        (U,\vec{Y})\in R_{\Diamond}
        \;\Longleftrightarrow\;
        F_{Y_1}\times F_{Y_2}\times F_{Y_3}\subseteq \Diamond^{-1}[U].
\end{equation}
Here $\Diamond^{-1}[U]=\{(a,b,c)\in A^3:\Diamond(a,b,c)\in U\}$.

The next result shows that the relation $R_{\Diamond}$ defined above faithfully represents the algebraic behavior of the operator $\Diamond$ in the Stone space of $\mathbf{A}$.

\begin{proposition}\label{stone-relation lemma}
Let $(\mathbf{A}, \Diamond)$ be a 2-additive 3-BAMO and let $a,b,c\in A$.
Define \[\Box(x,y,z)\defeq\neg \Diamond (\neg x,\neg y,\neg z).\] Then, the following hold:
\begin{enumerate}
    \item $\Diamond_{R_{\Diamond}}(\beta(a),\beta(b),\beta(c))=\beta(\Diamond(a,b,c))$;
    \item $\Box_{R_{\Diamond}}(\beta(a),\beta(b),\beta(c))=\beta(\Box(a,b,c))$;
    \item $\Box_{R_{\Diamond}}(\beta(a),\beta(b),\beta(c))
    =\Diamond_{R_{\Diamond}}(\beta(\neg a),\beta(\neg b),\beta(\neg c))^{c}$.
\end{enumerate}
\end{proposition}

\begin{proof}
(1) This follows from \eqref{def: R diamond}.
The proof is analogous to that of \citep[Theorem 5]{Celani2009}.

\smallskip

(2) Let $U\in \Box_{R_{\Diamond}}(\beta(a),\beta(b),\beta(c))$ and assume that
$U\notin \beta(\Box(a,b,c))$. Then
\[
\Diamond(\neg a,\neg b,\neg c)\in U.
\]
Since $\Diamond$ is increasing in all coordinates, we have
\[
[\neg a)\times [\neg b) \times [\neg c)
=F_{\beta(\neg a)}\times F_{\beta(\neg b)}\times F_{\beta(\neg c)}
\subseteq \Diamond^{-1}[U],
\]
so $(\beta(\neg a), \beta(\neg b), \beta(\neg c))\in R_{\Diamond}(U)$.

On the other hand, the fact that $U\in \Box_{R_{\Diamond}}(\beta(a),\beta(b),\beta(c))$ implies that $R_\Diamond(U)\subseteq L_{(\beta(a), \beta(b), \beta(c))}$. Therefore,
\[
(\beta(\neg a), \beta(\neg b), \beta(\neg c))\cap (\beta(a), \beta(b), \beta(c))\neq \emptyset,
\]
which is absurd. Hence
$\Box_{R_{\Diamond}}(\beta(a),\beta(b),\beta(c))
\subseteq \beta(\Box(a,b,c))$.

Conversely, let $U\in \beta(\Box(a,b,c))$ and suppose that $\vec{Y}\in R_{\Diamond}(U)$.
We show that $\vec{Y}\in L_{(\beta(a),\beta(b),\beta(c))}$.
Indeed, our assumptions entail that
(a) $\Diamond(\neg a, \neg b, \neg c)\notin U$, and
(b) $F_{Y_1}\times F_{Y_2}\times F_{Y_3}\subseteq \Diamond^{-1}[U]$.
From (a) and (b) it follows that
$(\neg a, \neg b, \neg c)\notin F_{Y_1}\times F_{Y_2}\times F_{Y_3}$,
and hence $Y_1\nsubseteq \beta(a)^c$ or $Y_2\nsubseteq \beta(b)^c$ or  $Y_3\nsubseteq \beta(c)^c$, and then at least one of the subsets $Y_1\cap\beta(a)$, $Y_2\cap\beta(b)$ and $Y_3\cap\beta(c)$ is non-empty.
Therefore, $\vec{Y}\cap (\beta(a),\beta(b),\beta(c))\neq \emptyset$,
so $\vec{Y}\in L_{(\beta(a),\beta(b),\beta(c))}$.
We conclude that $\beta(\Box(a,b,c))\subseteq \Box_{R_{\Diamond}}(\beta(a),\beta(b),\beta(c))$,
as required.

\smallskip

(3) The third equality stems from the following calculations:
\begin{align*}
\Box_{R_{\Diamond}}(\beta(a),\beta(b),\beta(c))
&=  \beta(\Box (a,b,c)) && \text{by (2)}\\
&=  \beta(\Diamond(\neg a,\neg b,\neg c))^{c} && \text{by the definition of $\Box$}\\
&=  \Diamond_{R_{\Diamond}}(\beta(\neg a),\beta(\neg b),\beta(\neg c))^{c} && \text{by (1).}\qedhere
\end{align*}
\end{proof}

By \citep[Proposition 13]{Celani2009}, the canonical space of a BAMO satisfies the descriptiveness conditions \eqref{DF1} and \eqref{DF2}. For our 2-additive 3-BAMOs, additivity in the first two coordinates also yields \eqref{DF3}. We give the details, adapting the ultrafilter-extension argument in the proof of \citep[Theorem 9]{Celani2009} only in these two coordinates.

\begin{proposition}\label{3BAMO to psiframe}
Let $\mathfrak{A}\defeq(\mathbf{A},\Diamond)$ be a 2-additive 3-BAMO and let $\tau_{\mathbf{A}}$ be the Stone topology on $\Ult(\mathbf{A})$. Then $\Uf(\mathfrak{A})=(\Ult(\mathbf{A}), \tau_{\mathbf{A}},  R_{\Diamond})$ is a PSI-frame.
\end{proposition}

\begin{proof}
\eqref{DF1} Every triple of clopen subsets of $\Ult(\mathbf{A})$ has the form $(\beta(a_1),\beta(a_2),\beta(a_3))$ for some $\vec{a}=(a_1,a_2,a_3)\in A^3$. By Proposition~\ref{stone-relation lemma}(1), its image under $\Diamond_{R_{\Diamond}}$ is $\beta(\Diamond(a_1,a_2,a_3))$, which is clopen. This proves \eqref{DF1}.

\smallskip

\eqref{DF2} Let $U\in \Ult(\mathbf{A})$. In the following, we denote $(\beta(a_1),\beta( a_2),\beta( a_3))$ by $\beta(\vec{a})$ and $(\beta(\neg a_1),\beta( \neg a_2), \beta(\neg a_3))$ by $\beta(\neg \vec{a})$.
Since such triples exhaust the triples of clopens, the right-hand side of \eqref{DF2} is
\[
D=\bigcap \left\{ L_{\beta(\vec{a})} \mid U \in \Box_{R_{\Diamond}}(\beta(\vec{a})),\;   \vec{a}\in A^3\right\}.
\]
By the definition, $U\in\Box_{R_{\Diamond}}(\beta(\vec{a}))$ implies $R_{\Diamond}(U)\subseteq L_{\beta(\vec{a})}$. Hence $R_{\Diamond}(U)\subseteq D$.

Suppose, for the sake of contradiction, that there exists $\vec{Y}\in D$ such that $\vec{Y}\notin R_{\Diamond}(U)$.
Then there exist elements $a_i\in F_{Y_i}$, for $i=1,2,3$, such that $\Diamond(\vec{a})\notin U$.
Consequently, $\vec{Y}\cap \beta(\neg \vec{a})=\emptyset$, and since $U$ is an ultrafilter, we have $\Box(\neg \vec{a})\in U$. By
Proposition~\ref{stone-relation lemma}(2) it is the case that $U\in\Box_{R_{\Diamond}}(\beta(\neg\vec{a}))$. Thus $L_{\beta(\neg\vec{a})}$ is one of the families in the intersection defining $D$. However, $\vec{Y}\notin L_{\beta(\neg\vec{a})}$. This contradicts $\vec{Y}\in D$. Hence $D\subseteq R_{\Diamond}(U)$, proving \eqref{DF2}.

\smallskip

\eqref{DF3} Suppose that $\vec{Y}\in R_\Diamond(U)$. Then $F_{Y_1}\times F_{Y_2}\times F_{Y_3}\subseteq \Diamond^{-1}[U]$. We will extend $F_{Y_1}$ and then $F_{Y_2}$ to ultrafilters, preserving the product inclusion at each step.

Consider the family
\[
\mathcal{F}_1\defeq\{F\in \mathrm{Fi}(\mathbf{A}) :
F_{Y_1}\subseteq F \text{ and } F\times F_{Y_2}\times F_{Y_3}\subseteq \Diamond^{-1}[U]\}
\]
ordered by set-theoretical inclusion. It is non-empty because $F_{Y_1}\in\mathcal{F}_1$. Every member of $\mathcal{F}_1$ is proper: if $0\in F\in\mathcal{F}_1$, then $\Diamond(0,1,1)=0\in U$ by \eqref{MO1}, a contradiction. If $\mathcal{H}\subseteq\mathcal{F}_1$ is a non-empty chain, its union $F^*=\bigcup\mathcal{H}$ is a filter containing $F_{Y_1}$. For any $a\in F^*$, there is $F\in\mathcal{H}$ containing $a$; hence $\Diamond(a,b,c)\in U$ for every $b\in F_{Y_2}$ and $c\in F_{Y_3}$. Thus $F^*\in\mathcal{F}_1$ and is an upper bound of the chain. So, by the Kuratowski--Zorn lemma, there is a maximal member $V$ of $\mathcal{F}_1$.

We prove that $V$ is an ultrafilter. Let $a\in A$ be such that $a\notin V$. Write $\mathsf{Fg}^{\mathbf{A}}(S)$ for the filter generated by $S\subseteq A$. Since $V\subsetneq\mathsf{Fg}^{\mathbf{A}}(V\cup\{a\})$, maximality gives $b\in\mathsf{Fg}^{\mathbf{A}}(V\cup\{a\})$, $x_1\in F_{Y_2}$ and $y_1\in F_{Y_3}$ such that $\Diamond(b,x_1,y_1)\notin U$.

We show that $\mathsf{Fg}^{\mathbf{A}}(V\cup\{\neg a\})\in\mathcal{F}_1$. Take arbitrary $c\in\mathsf{Fg}^{\mathbf{A}}(V\cup\{\neg a\})$, $x_2\in F_{Y_2}$ and $y_2\in F_{Y_3}$. There are $f_1,f_2\in V$ with $f_1\wedge\neg a\leq c$ and $f_2\wedge a\leq b$. As the three filters are closed under meets, $\Diamond(f_1\wedge f_2,x_1\wedge x_2,y_1\wedge y_2)\in U$. By \eqref{MO2},
\begin{align*}
    \Diamond(f_1\wedge f_2, x_1\wedge x_2, y_1\wedge y_2)
    &= \Diamond(f_1\wedge f_2\wedge(a\vee \neg a), x_1\wedge x_2, y_1\wedge y_2)\\
    &= \Diamond((f_1\wedge f_2\wedge a)\vee (f_1\wedge f_2\wedge \neg a), x_1\wedge x_2, y_1\wedge y_2)\\
    &= \Diamond(f_1\wedge f_2\wedge a, x_1\wedge x_2, y_1\wedge y_2) \\
    &\qquad\vee\; \Diamond(f_1\wedge f_2\wedge \neg a, x_1\wedge x_2, y_1\wedge y_2).
\end{align*}

Thus, either $\Diamond(f_1\wedge f_2\wedge a, x_1\wedge x_2, y_1\wedge y_2)\in U$ or $\Diamond(f_1\wedge f_2\wedge \neg a, x_1\wedge x_2, y_1\wedge y_2)\in U$. In the first case, by monotonicity, it follows that $\Diamond(b,x_1,y_1)\in U$, which is a contradiction. Therefore, $\Diamond(f_1\wedge f_2\wedge \neg a, x_1\wedge x_2, y_1\wedge y_2)\in U$, and by monotonicity, $\Diamond(c,x_2,y_2)\in U$. Hence,
\[
\mathsf{Fg}^{\mathbf{A}}(V\cup\{\neg a\})\times F_{Y_2}\times F_{Y_3}
\subseteq \Diamond^{-1}[U],
\]
and by maximality of $V$, we conclude that $\neg a\in V$.

Since $V$ is proper and contains $a$ or $\neg a$ for every $a\in A$, it is an ultrafilter. Moreover, $F_{Y_1}\subseteq V$ gives $V\in Y_{F_{Y_1}}=Y_1$.

We now keep this ultrafilter $V$ fixed and consider
\[
\mathcal{F}_2\defeq\{F\in\mathrm{Fi}(\mathbf{A}):F_{Y_2}\subseteq F
\text{ and }V\times F\times F_{Y_3}\subseteq\Diamond^{-1}[U]\}.
\]
This family contains $F_{Y_2}$. Its members are proper by \eqref{MO1}, since $0\in F$ would imply $\Diamond(1,0,1)=0\in U$. The union of any non-empty chain in $\mathcal{F}_2$ remains in $\mathcal{F}_2$, by the same elementwise argument as for $\mathcal{F}_1$. Hence, $\mathcal{F}_2$ has a maximal member $W$.

The maximality argument for $V$ applies to $W$ in the second coordinate as well, with \eqref{MO3} in place of \eqref{MO2}. Explicitly, if $a\notin W$, maximality supplies $b\in\mathsf{Fg}^{\mathbf{A}}(W\cup\{a\})$, $x_1\in V$ and $y_1\in F_{Y_3}$ with $\Diamond(x_1,b,y_1)\notin U$. Given $c\in\mathsf{Fg}^{\mathbf{A}}(W\cup\{\neg a\})$, $x_2\in V$ and $y_2\in F_{Y_3}$, choose $f_1,f_2\in W$ with $f_1\wedge\neg a\leq c$ and $f_2\wedge a\leq b$. Then
\[
\begin{aligned}
\Diamond(x_1\wedge x_2,f_1\wedge f_2,y_1\wedge y_2)
={}&\Diamond(x_1\wedge x_2,f_1\wedge f_2\wedge a,y_1\wedge y_2)\\
&\vee\Diamond(x_1\wedge x_2,f_1\wedge f_2\wedge\neg a,y_1\wedge y_2).
\end{aligned}
\]
The left-hand side belongs to $U$, while the first join does not, since it is below $\Diamond(x_1,b,y_1)$. Thus, the second join belongs to $U$, and monotonicity entails that $\Diamond(x_2,c,y_2)\in U$. It follows that $\mathsf{Fg}^{\mathbf{A}}(W\cup\{\neg a\})\in\mathcal{F}_2$. By maximality, $\neg a\in W$. Consequently, $W$ is an ultrafilter.

Finally, $F_{Y_2}\subseteq W$ implies $W\in Y_2$. The inclusion $V\times W\times F_{Y_3}\subseteq\Diamond^{-1}[U]$ is exactly $(\{V\},\{W\},Y_3)\in R_\Diamond(U)$. This proves \eqref{DF3} and completes the proof.
\end{proof}

\section{PSI-spaces}\label{sec:psi-spaces}

Having established the representation of 2-additive 3-BAMOs,
we now introduce PSI-spaces by imposing four additional conditions
on PSI-frames. As shown below, these conditions ensure that the
associated clopen algebra is a PSI-Algebra.

\begin{definition}\label{def: pseudo inference-space}
    A \emph{pseudo-inference space} (or simply \emph{PSI-space}) is a PSI-frame $(X,  \tau, R)$ that satisfies the following four conditions for all $Y_1,Y_2,Y_3,Z,W\in\mathscr{C}(X)$:
    \begin{gather}
    R^{-1}(Y_1,Y_2,Y_3)\subseteq R^{-1}(Y_1,Y_2,Z)\cup R^{-1}(Y_1,Y_2,W) \cup R^{-1}(\Cl(Z^c),\Cl(W^c),Y_3)\,,\tag{PIF1}\label{PIF1}\\
        Y_1\cap Y_3 = \emptyset \Rightarrow R^{-1}(Y_1,Y_2,Y_3)=\emptyset\,,\tag{PIF2}\label{PIF2}\\
        Y_1\cap Y_2 \subseteq R^{-1}(Y_1, Y_1, Y_2)\,,\tag{PIF3}\label{PIF3}\\
        R^{-1}(Y_1,Y_2,Y_3)\subseteq R^{-1}(Y_2,Y_1,Y_3).\tag{PIF4}\label{PIF4}
    \end{gather}\QED
\end{definition}

Together with Corollary \ref{psi-frame to 3BAMO}, this definition yields a direct algebraic counterpart of PSI-spaces:

\begin{proposition}\label{space to algebra}
    Let $(X,\tau, R)$ be a PSI-space. Then, the pair $(\CO(X),\Diamond_{R})$ is a PSI-algebra.
\end{proposition}

\begin{proof}
By Corollary~\ref{psi-frame to 3BAMO}, $(\CO(X),\Diamond_{R})$ is a 2-additive 3-BAMO. We now verify conditions \eqref{PI1}--\eqref{PI4}.
By Proposition~\ref{psi-frames are m-frames}, we have
\[
\Diamond_R(A,B,F)=R^{-1}(A,B,F)
\]
for all $A,B,F\in\CO(X)$.

Let $A,B,D,E,F\in\CO(X)$.
Applying \eqref{PIF1} with
$Y_1=A$, $Y_2=B$, $Y_3=F$, $Z=D^c$ and $W=E^c$,
and observing that
$\Cl((D^c)^c)=D$ and
$\Cl((E^c)^c)=E$, we obtain
\[
\Diamond_R(A,B,F)
\subseteq
\Diamond_R(A,B,D^c)\cup
\Diamond_R(A,B,E^c)\cup
\Diamond_R(D,E,F).
\]
This is \eqref{PI1} in the Boolean algebra of clopens.

Since $A\cap A^c=\emptyset$, condition \eqref{PIF2} gives us
\[
\Diamond_R(A,B,A^c)=R^{-1}(A,B,A^c)=\emptyset,
\]
which proves \eqref{PI2}.
Likewise, \eqref{PIF3} gives us
\[
A\cap F\subseteq R^{-1}(A,A,F)=\Diamond_R(A,A,F),
\]
proving \eqref{PI3}.
Finally, \eqref{PIF4} yields
\[
\Diamond_R(A,B,F)\subseteq\Diamond_R(B,A,F),
\]
which is \eqref{PI4}.
\end{proof}

We next establish the correspondence between the additional
axioms of PSI-Algebras and the conditions defining PSI-spaces.
For a 2-additive 3-BAMO $\mathfrak{A}$, the following lemmas show that
$\mathfrak{A}$ satisfies \eqref{PI1}--\eqref{PI4} if and only if
its ultrafilter frame $\Uf(\mathfrak{A})$ satisfies
\eqref{PIF1}--\eqref{PIF4}.

\begin{lemma}\label{Relational PI1}
If $\mathfrak{A}\defeq(\mathbf{A},\Diamond)$ is a 2-additive 3-BAMO, then $\mathfrak{A}$ satisfies \eqref{PI1}\footnote{Each condition for PSI-spaces is formulated in an open form (without quantifiers). When we write that a frame $\frF$ satisfies an open condition we mean that $\frF$ satisfies its universal closure.} iff $\Uf(\mathfrak{A})$ satisfies \eqref{PIF1}.
\end{lemma}
\begin{proof}
Assume that \eqref{PIF1} holds in $\Uf(\mathfrak{A})$. By the properties of $\beta$, for every $x \in A$ we have $\Cl(\beta(\neg x)^c) = \beta( x)$. So, for all $a, b, d,e, f \in A $,
\[
\begin{aligned}
R_{\Diamond}^{-1}(\beta(a),\beta(b),\beta(f)) \subseteq\;&
R_{\Diamond}^{-1}(\beta(a),\beta(b),\beta(\neg d)) \cup {}\\
& R_{\Diamond}^{-1}(\beta(a),\beta(b),\beta(\neg e)) \cup
R_{\Diamond}^{-1}(\beta( d),\beta( e),\beta(f)).
\end{aligned}
\]
Using propositions~\ref{psi-frames are m-frames} and~\ref{stone-relation lemma}(1), we obtain
\[
\beta(\Diamond(a,b,f))
\subseteq
\beta(\Diamond(a,b,\neg d))
\cup\beta(\Diamond(a,b,\neg e))\cup\beta(\Diamond(d,e,f)).
\]
As $\beta$ preserves finite joins and reflects the order, the inclusion above is equivalent to
\[
\Diamond(a,b,f)\leq
\Diamond(a,b,\neg d)\vee
\Diamond(a,b,\neg e)\vee
\Diamond(d,e,f),
\]
which is \eqref{PI1}.

\smallskip

For the converse, assume that $\mathfrak{A}$ satisfies \eqref{PI1}.
Fix closed non-empty sets of ultrafilters $Y_1, Y_2, Y_3, Z,$ and $W$.
Let $U$ be an ultrafilter of $\mathbf{A}$ such that
(a) $U \in R_{\Diamond}^{-1}(Y_1,Y_2,Y_3)$,
(b) $U \notin R_{\Diamond}^{-1}(Y_1,Y_2,Z)$, and
(c) $U \notin R_{\Diamond}^{-1}(Y_1,Y_2,W)$.
We show that $U \in \RD^{-1}(\Cl(Z^c), \Cl(W^c), Y_3)$.

To this end, let $p \in F_{\Cl(Z^c)}$, $q \in F_{\Cl(W^c)}$, and $r \in F_{Y_3}$.  By the definition,
\[
Z^c\subseteq\Cl(Z^c)\subseteq\beta(p),
\qquad
W^c\subseteq\Cl(W^c)\subseteq\beta(q).
\]
Taking complements,
we obtain
\[
\beta(\neg p)\subseteq Z,
\qquad
\beta(\neg q)\subseteq W.
\]

By (b), there exist $a \in F_{Y_1}$, $b \in F_{Y_2}$, and $c \in F_Z$ such that $\Diamond(a,b,c) \notin U$.
Similarly, by (c), there exist $x \in F_{Y_1}$, $y \in F_{Y_2}$, and $l \in F_W$ such that $\Diamond(x,y,l) \notin U$.
Hence, putting $s\defeq a \wedge x$ and $t\defeq b \wedge y$, by the monotonicity of $\Diamond$ in each coordinate we obtain
(d) $\neg\Diamond(s,t,c) \wedge \neg\Diamond(s,t,l) \in U$.

Moreover, since $\beta(\neg p) \subseteq Z \subseteq \beta(c)$ and $\beta(\neg q) \subseteq W \subseteq \beta(l)$, it follows that $\neg p \leq c$ and $\neg q \leq l$.
Applying the monotonicity of $\Diamond$ to (d), we get
$\neg\Diamond(s,t,\neg p) \wedge \neg\Diamond(s,t,\neg q) \in U$.
Together with (a), this yields
$\Diamond(s,t,r) \wedge \neg\Diamond(s,t,\neg p) \wedge \neg\Diamond(s,t,\neg q) \in U$.
Therefore, by \eqref{PI1}, we obtain $\Diamond(p,q,r) \in U$.
This proves that $U \in \RD^{-1}(\Cl(Z^c),\Cl(W^c),Y_3)$, which in turn entails that \eqref{PIF1} holds, as claimed.
\end{proof}

\begin{lemma}\label{Relational PI2}
If $\mathfrak{A}$ is a 2-additive 3-BAMO, then $\mathfrak{A}$ satisfies \eqref{PI2} iff $\Uf(\mathfrak{A})$ satisfies \eqref{PIF2}.
\end{lemma}
\begin{proof}
    For the right-to-left direction, assume that $\Uf(\mathfrak{A})$ satisfies \eqref{PIF2}. Let $a,b\in A$. Since
$\beta(a)\cap\beta(\neg a)=\emptyset$, we have
\[
R_\Diamond^{-1}(\beta(a),\beta(b),\beta(\neg a))
=\emptyset.
\]
By Proposition~\ref{psi-frames are m-frames} and
Proposition~\ref{stone-relation lemma}(1), respectively,
\[
\begin{aligned}
R_\Diamond^{-1}(\beta(a),\beta(b),\beta(\neg a))
&=\Diamond_{R_\Diamond}(\beta(a),\beta(b),\beta(\neg a))\\
&=\beta(\Diamond(a,b,\neg a)).
\end{aligned}
\]
Thus $\beta(\Diamond(a,b,\neg a))=\emptyset=\beta(0)$.
The injectivity of $\beta$ entails $\Diamond(a,b,\neg a)=0$,
which is \eqref{PI2}.

\smallskip

For the left-to-right direction, suppose that $\mathfrak{A}$ satisfies \eqref{PI2}.
We proceed by contraposition. Let $U$ be an ultrafilter in $\RD^{-1}(Y_1,Y_2,Y_3)$, that is,
$F_{Y_{1}} \times F_{Y_{2}} \times F_{Y_{3}} \subseteq \Diamond^{-1}[U]$.
For any triple $(a,b,f)$ from this product, we have $\Diamond(a,b,f) \neq 0$.
Therefore, by \eqref{PI2} and Proposition~\ref{quasiecuation PI2}, for every pair of elements $a \in F_{Y_1}$ and $f \in F_{Y_3}$, it follows that $a \wedge f \neq 0$.

Let $G\defeq F_{Y_1}\vee F_{Y_3}$ be the filter generated by
$F_{Y_1}\cup F_{Y_3}$. Every element of $G$ lies above
$a\wedge f$ for some $a\in F_{Y_1}$ and $f\in F_{Y_3}$.
Therefore, the preceding argument shows that $0\notin G$,
so $G$ is proper.

An ultrafilter contains $G$ if and only if it contains
both $F_{Y_1}$ and $F_{Y_3}$. Consequently,
\[
Y_G=Y_{F_{Y_1}}\cap Y_{F_{Y_3}}=Y_1\cap Y_3.
\]
Applying the inverse correspondence also gives
$G=F_{Y_1\cap Y_3}$.
Since the proper filter $G$ extends to an ultrafilter,
$Y_G\neq\emptyset$. Thus $Y_1\cap Y_3\neq\emptyset$,
which proves \eqref{PIF2}.
\end{proof}

\begin{remark}
The condition \eqref{PIF2} cannot be replaced by the simpler condition
\begin{equation}\label{Ecua Rem}
    R^{-1}(Y_1,Y_2,\Cl(Y_1^c))=\emptyset.
\end{equation}
For the canonical frame of a 2-additive 3-BAMO, condition
\eqref{Ecua Rem} implies \eqref{PI2}.
Indeed, take $Y_1\defeq \beta(a)$ and $Y_2\defeq\beta(b)$.
Since $\beta(a)$ is clopen and $\beta$ preserves complements,
\[
\Cl(\beta(a)^c)=\beta(\neg a).
\]
Thus, by Proposition~\ref{psi-frames are m-frames} and
Proposition~\ref{stone-relation lemma}(1),
\[
\begin{aligned}
\beta(\Diamond(a,b,\neg a))
&=\Diamond_{R_\Diamond}(\beta(a),\beta(b),\beta(\neg a))\\
&=R_\Diamond^{-1}(\beta(a),\beta(b),\beta(\neg a))
=\emptyset.
\end{aligned}
\]
The injectivity of $\beta$ gives $\Diamond(a,b,\neg a)=0$.
However, the converse implication fails, as the following
example shows.

Let $X$ be $\Ult(\mathcal{P}(\Nat))$, the Stone space of the complete algebra $\mathcal{P}(\Nat)$.
$X$ is a realization of the Stone--\v{C}ech
compactification $\sigma\Nat$ of the discrete space
$\Nat$, where we identify each $n\in\Nat$ with the
principal ultrafilter
\[
U_n\defeq\{A\subseteq\Nat:n\in A\}
\]
and we add uncountably many new points in the form of non-principal ultrafilters.
Since $\beta(\{n\})=\{U_n\}$, the copy of $\Nat$
is open in $X$. It is also dense: every non-empty basic
clopen $\beta(A)$ contains $U_n$ for some $n\in A$.

Let us consider $Y\defeq X\setminus\Nat$, the closed set of non-principal
ultrafilters.
Since $Y^c=\Nat$ is dense in $X$, we have
\[
\Cl(Y^c)=\Cl(\Nat)=X
\]
(in particular, $Y$ is not open). Any ultrafilter of $\power(\Nat)$ is non-principal if and only if it contains
every cofinite set (equivalently, if it extends the cofinite---or Fr\'echet---filter). So, we obtain
\[
F_{Y}=\{A\subseteq\Nat:A\text{ is cofinite}\}
\qquad\text{and}\qquad
F_{\Cl(Y^c)}=\{\Nat\}\,.
\]

Consider the PSI-Algebra \((\mathcal{P}(\Nat),\Diamond)\) with \(\Diamond(A,B,C)\defeq A\cap B\cap C\) (cf. Example \ref{smallest Diamond}).
We have that
\[
U\in R_\Diamond^{-1}\bigl(Y,X,\Cl(Y^c)\bigr)
\Longleftrightarrow F_Y\subseteq U,
\]
and thus \(U\) must contain every cofinite subset of \(\Nat\); equivalently, \(U\) is a non-principal ultrafilter. Hence
\[
R_\Diamond^{-1}\bigl(Y,X,\Cl(Y^c)\bigr)=Y\neq\varnothing\,,
\]
and therefore \eqref{Ecua Rem} fails in this case, although $(\mathcal{P}(\Nat),\Diamond)$ satisfies \eqref{PI2}.\QED
\end{remark}

\begin{lemma}\label{lem:preimage-simpler}
    Let $(\mathbf{A},\Diamond)$ be a 2-additive 3-BAMO. For all closed sets $X,Y$ of ultrafilters of $\mathbf{A}$ and for any ultrafilter $U$, we have
\[
F_X \times F_X \times F_Y \subseteq \Diamond^{-1}[U] \quad \text{if and only if} \quad
\Diamond(a,a,b) \in U \ \text{for all } a \in F_X \text{ and } b \in F_Y.
\]
\end{lemma}
\begin{proof}
    The implication from left to right is obvious. For the reverse one, take $a,b\in F_{X}$ and $c\in F_Y$. Since $F_X$ is a filter, we have that $a\wedge b\in F_{X}$, and thus $\Diamond(a\wedge b,a\wedge b,c)\in U$. As the diamond is a monotonic operator and $U$ is upward closed, we obtain that $\Diamond(a,b,c)\in U$.
\end{proof}

\begin{lemma}\label{Relational PI3}
If $\mathfrak{A}$ is a 2-additive 3-BAMO, then $\mathfrak{A}$ satisfies \eqref{PI3} iff $\Uf(\mathfrak{A})$ satisfies \eqref{PIF3}.
\end{lemma}
\begin{proof}
    First, let us assume \eqref{PIF3} for $\Uf(\mathfrak{A})$. Let $a,f\in A$.  An application of \eqref{PIF3} to the clopen sets
$\beta(a)$ and $\beta(f)$ results in
\[
\beta(a)\cap\beta(f)
\subseteq R_\Diamond^{-1}(\beta(a),\beta(a),\beta(f)).
\]
Since $\beta$ preserves meets,
$\beta(a)\cap\beta(f)=\beta(a\wedge f)$.
Moreover, by Proposition~\ref{psi-frames are m-frames}
and Proposition~\ref{stone-relation lemma}(1),
\[
\begin{aligned}
R_\Diamond^{-1}(\beta(a),\beta(a),\beta(f))
&=\Diamond_{R_\Diamond}(\beta(a),\beta(a),\beta(f))\\
&=\beta(\Diamond(a,a,f)).
\end{aligned}
\]
Thus $\beta(a\wedge f)\subseteq\beta(\Diamond(a,a,f))$.
Because $\beta$ reflects the order, we obtain
$a\wedge f\leq\Diamond(a,a,f)$, thus proving \eqref{PI3}.

\smallskip

Second, suppose that \eqref{PI3} obtains in $\mathfrak{A}$. To derive \eqref{PIF3} for $\Uf(\mathfrak{A})$, pick an arbitrary $U\in Y_1\cap Y_2$ with the aim of showing that $\RD(U,Y_1,Y_1,Y_2)$. To this end, let $a\in F_{Y_1}$ and $f\in F_{Y_2}$. By the choice of $U$, $a\wedge f$ is in $U$, so by \eqref{PI3} we have that $\Diamond(a,a,f)\in U$. It follows from Lemma~\ref{lem:preimage-simpler} that $F_{Y_{1}}\times F_{Y_{1}}\times F_{Y_{2}}\subseteq\Diamond^{-1}[U]$, that is $U\in\RD^{-1}(Y_1,Y_1,Y_2)$.
\end{proof}

\begin{lemma}\label{Relational PI4}
If $\mathfrak{A}$ is a 2-additive 3-BAMO, then $\mathfrak{A}$ satisfies \eqref{PI4} iff $\Uf(\mathfrak{A})$ satisfies \eqref{PIF4}.
\end{lemma}
\begin{proof}
    Firstly, assume \eqref{PI4} for $\mathfrak{A}$. Let ($\dagger$) $\RD(U,Y_1,Y_2,Y_3)$ and take $(b,a,f)\in F_{Y_2}\times F_{Y_1}\times F_{Y_3}$. Thus $(a,b,f)\in F_{Y_1}\times F_{Y_2}\times F_{Y_3}$, and by ($\dagger$) we have that $\Diamond(a,b,f)\in U$. In consequence, since $U$ is upward closed, by \eqref{PI4} we obtain that $\Diamond(b,a,f)\in U$. It follows that $\RD(U,Y_2,Y_1,Y_3)$ and in consequence $\Uf(\mathfrak{A})$ satisfies \eqref{PIF4}.

    \smallskip

    If \eqref{PIF4} holds for $\Uf(\mathfrak{A})$, then for any $a,b,f\in A$ we have that
    \[
    \RD^{-1}(\beta(a),\beta(b),\beta(f))\subseteq \RD^{-1}(\beta(b),\beta(a),\beta(f))\,,
    \]
    since each $\beta(x)$ is clopen. By Proposition~\ref{psi-frames are m-frames} and
Proposition~\ref{stone-relation lemma}(1), we obtain
\[
\beta(\Diamond(a,b,f))
\subseteq\beta(\Diamond(b,a,f)).
\]
Since $\beta$ reflects the order, the inclusion yields
\[
\Diamond(a,b,f)\leq\Diamond(b,a,f),
\]
which is \eqref{PI4}.
\end{proof}

Below, we combine lemmas \ref{Relational PI1}, \ref{Relational PI2}, \ref{Relational PI3} and \ref{Relational PI4}, in order to provide a characterization of PSI-Algebras by means of PSI-frames.

\begin{theorem}\label{Psi-algebras as ultrafilter frames}
        Let $\mathfrak{A}=(\mathbf{A},\Diamond)$ be a 2-additive 3-BAMO, and let $\Uf(\mathfrak{A})=(\Ult(\mathbf{A}),\tau_{\mathbf{A}}, \RD)$ be its associated PSI-frame. Then, for every $i\in\{1,2,3,4\}$,
        \[
            \text{$\mathfrak{A}$ satisfies \textup{(PI$i$)}}\quad\text{iff}\quad\text{$\Uf(\mathfrak{A})$ satisfies \textup{(PIF$i$)}}\,.
        \]
        Consequently, $\mathfrak{A}$ is a Pseudo-Inference Algebra iff $\Uf(\mathfrak{A})$ is a PSI-space.
\end{theorem}

\section{Morphisms and dualities}\label{sec:morphisms-dualities}

We have shown how to associate a PSI-space with each
PSI-Algebra and a PSI-Algebra with each PSI-space.
We now turn to morphisms, relating the algebraic notions
of PSI semi-homomorphism, hemi-homomorphism and homomorphism
to suitable continuous maps between PSI-spaces.
These correspondences will be used to establish
the three dual equivalences below.

Let \((\mathbf{A},\Diamond_{\mathbf{A}})\) and \((\mathbf{B},\Diamond_{\mathbf{B}})\) be PSI-algebras.
A Boolean homomorphism \(h\colon \mathbf{A}\to \mathbf{B}\) is called a
\emph{Pseudo-Inference semi-homomorphism} (PSI semi-homomorphism) if,
for all \(a,b,c\in A\),
\[
h(\Diamond_{\mathbf{A}}(a,b,c)) \leq \Diamond_{\mathbf{B}}(h(a),h(b),h(c)).
\]

Dually, \(h\) is a \emph{Pseudo-Inference hemi-homomorphism} (PSI hemi-homomorphism) if
\[
\Diamond_{\mathbf{B}}(h(a),h(b),h(c)) \leq h(\Diamond_{\mathbf{A}}(a,b,c)),
\]
for all \(a,b,c\in A\). When equality holds,
\[
h(\Diamond_{\mathbf{A}}(a,b,c)) = \Diamond_{\mathbf{B}}(h(a),h(b),h(c)),
\]
we say that \(h\) is a \emph{Pseudo-Inference homomorphism} (PSI homomorphism).

For a function \(f\colon X\to Y\) and subsets \(Z\subseteq X\), \(W\subseteq Y\),
we write \(f[Z]\) for the image of \(Z\) and \(f^{-1}[W]\) for the preimage of \(W\).

\begin{lemma}\label{morphisms to maps}
    Let $(\mathbf{A},\Diamond_{\mathbf{A}})$ and $(\mathbf{B},\Diamond_{\mathbf{B}})$ be two PSI-Algebras and let $h\colon\mathbf{A} \to \mathbf{B}$ be a homomorphism of Boolean algebras. Let $f\colon\Uf(\mathbf{B})\to \Uf(\mathbf{A})$ be the map $U\mapsto h^{-1}[U]$. Then:
    \begin{enumerate}
    \item The following are equivalent:
    \begin{enumerate}
        \item $h$ is a PSI semi-homomorphism.
        \item For every $a,b,c\in A$, \[f^{-1}[R_{\Diamond_{\mathbf{A}}}^{-1}(\beta_{\mathbf{A}}(a),\beta_{\mathbf{A}}(b),\beta_{\mathbf{A}}(c))]\subseteq R_{\Diamond_{\mathbf{B}}}^{-1}(\beta_{\mathbf{B}}(h(a)),\beta_{\mathbf{B}}(h(b)),\beta_{\mathbf{B}}(h(c))).\]
        \item For every $U\in \Uf(\mathbf{B})$ and $\vec{Z}\in \mathscr{C}_{0}(\Uf(\mathbf{A}))^3$
            \[(f(U),\vec{Z})\in R_{\Diamond_{\mathbf{A}}}\; \Rightarrow\; \exists \vec{Y}\in R_{\Diamond_{\mathbf{B}}}(U), \; f[\vec{Y}]\subseteq \vec{Z},\]
            where $f[\vec{Y}]\defeq(f[Y_1],f[Y_2],f[Y_3])$.
    \end{enumerate}
        \item The following are equivalent:
        \begin{enumerate}
            \item $h$ is a PSI hemi-homomorphism.
            \item For every $a,b,c\in A$, \[ R_{\Diamond_{\mathbf{B}}}^{-1}(\beta_{\mathbf{B}}(h(a)),\beta_{\mathbf{B}}(h(b)),\beta_{\mathbf{B}}(h(c)))\subseteq f^{-1}[R_{\Diamond_{\mathbf{A}}}^{-1}(\beta_{\mathbf{A}}(a),\beta_{\mathbf{A}}(b),\beta_{\mathbf{A}}(c))].\]
            \item For every $U\in \Uf(\mathbf{B})$ and $\vec{Y}\in \mathscr{C}_{0}(\Uf(\mathbf{B}))^3$
            \[(U,\vec{Y})\in R_{\Diamond_{\mathbf{B}}}\; \Rightarrow (f(U),f[\vec{Y}])\in R_{\Diamond_{\mathbf{A}}}.\]
        \end{enumerate}
    \end{enumerate}
\end{lemma}
\begin{proof}
    We start by recalling that, by Stone duality, it is the case that    \begin{equation}\label{Property Stone duality}
        f^{-1}[\beta_{\mathbf{A}}(a)]=\beta_{\mathbf{B}}(h(a)),
        \end{equation}
        for every $a\in A$. Now we proceed with the proof.

        \smallskip

        (1) (a) $\Leftrightarrow$ (b).
Let $a,b,c\in A$.
By Proposition~\ref{psi-frames are m-frames},
Proposition~\ref{stone-relation lemma}(1) and
\eqref{Property Stone duality}, we have
\[
\begin{aligned}
&f^{-1}[R_{\Diamond_{\mathbf{A}}}^{-1}
(\beta_{\mathbf{A}}(a),\beta_{\mathbf{A}}(b),
 \beta_{\mathbf{A}}(c))]\\
&\qquad=f^{-1}[\beta_{\mathbf{A}}
(\Diamond_{\mathbf{A}}(a,b,c))]\\
&\qquad=\beta_{\mathbf{B}}
(h(\Diamond_{\mathbf{A}}(a,b,c))).
\end{aligned}
\]
Likewise,
\[
\begin{aligned}
&R_{\Diamond_{\mathbf{B}}}^{-1}
(\beta_{\mathbf{B}}(h(a)),\beta_{\mathbf{B}}(h(b)),
 \beta_{\mathbf{B}}(h(c)))\\
&\qquad=\beta_{\mathbf{B}}
(\Diamond_{\mathbf{B}}(h(a),h(b),h(c))).
\end{aligned}
\]
Since $\beta_{\mathbf{B}}$ reflects the order,
the inclusion in (b) is equivalent to
\[
h(\Diamond_{\mathbf{A}}(a,b,c))
\leq \Diamond_{\mathbf{B}}(h(a),h(b),h(c)).
\]
Requiring this for every $a,b,c\in A$ is precisely
the definition of a PSI semi-homomorphism.

        \smallskip

        (c) $\Rightarrow$ (b) If $U\in f^{-1}[R_{\Diamond_{\mathbf{A}}}^{-1}(\beta_{\mathbf{A}}(a_1),\beta_{\mathbf{A}}(a_2),\beta_{\mathbf{A}}(a_3))]$, then \[(f(U),(\beta_{\mathbf{A}}(a_1),\beta_{\mathbf{A}}(a_2),\beta_{\mathbf{A}}(a_3)))\in R_{\Diamond_{\mathbf{A}}}.\]
        By (c), there exists $\vec{Y}\in R_{\Diamond_{\mathbf{B}}}(U)$ such that $f[\vec{Y}]\subseteq (\beta_{\mathbf{A}}(a_1),\beta_{\mathbf{A}}(a_2),\beta_{\mathbf{A}}(a_3))$.  Therefore,
        \begin{equation}\label{filters aux}
          F_{Y_1}\times F_{Y_2}\times F_{Y_3}\subseteq \Diamond_{\mathbf{B}}^{-1}[U]
        \end{equation}
        and $f[Y_i]\subseteq \beta_{\mathbf{A}}(a_i)$, for every $1\leqslant i\leqslant 3$. In order to prove our claim, we need to show that
        \[F_{f^{-1}[\beta_{\mathbf{A}}(a_1)]}\times F_{f^{-1}[\beta_{\mathbf{A}}(a_2)]} \times F_{f^{-1}[\beta_{\mathbf{A}}(a_3)]}\subseteq \Diamond_{\mathbf{B}}^{-1}[U].\]
        To this end, let $x_1,x_2,x_3\in B$ be such that $f^{-1}[\beta_{\mathbf{A}}(a_i)] \subseteq \beta_{\mathbf{B}}(x_i)$, for every $1\leqslant i\leqslant 3$. Then, it follows that $Y_i\subseteq f^{-1}[\beta_{\mathbf{A}}(a_i)]\subseteq \beta_{\mathbf{B}}(x_i)$. Since $(x_1,x_2,x_3)\in F_{f^{-1}[\beta_{\mathbf{A}}(a_1)]}\times F_{f^{-1}[\beta_{\mathbf{A}}(a_2)]} \times F_{f^{-1}[\beta_{\mathbf{A}}(a_3)]}$, by (\ref{filters aux}) we can conclude that  $\Diamond_{\mathbf{B}}(x_1,x_2,x_3)\in U$. Hence, because of the latter and (\ref{Property Stone duality}), we can say that $U\in R_{\Diamond_{\mathbf{B}}}^{-1}(\beta_{\mathbf{B}}(h(a_1)),\beta_{\mathbf{B}}(h(a_2)),\beta_{\mathbf{B}}(h(a_3)))$, as claimed.

        \smallskip

        (b) $\Rightarrow$ (c). Let $U\in \Uf(\mathbf{B})$ and $\vec{Z}\in \mathscr{C}_{0}(\Uf(\mathbf{A}))^3$ be such that $(f(U),\vec{Z})\in R_{\Diamond_{\mathbf{A}}}$. Observe that by (b) and the monotonicity of $R_{\Diamond_{\mathbf{A}}}^{-1}$, we have that
        \begin{gather*}
         f^{-1}[R_{\Diamond_{\mathbf{A}}}^{-1}(\vec{Z})]\subseteq R_{\Diamond_{\mathbf{B}}}^{-1}(\beta_{\mathbf{B}}(h(a_1)),\beta_{\mathbf{B}}(h(a_2)),\beta_{\mathbf{B}}(h(a_3)))\,,
        \end{gather*}
        for every $a_i\in F_{Z_i}$. Thus, since $U\in f^{-1}[R_{\Diamond_{\mathbf{A}}}(\vec{Z})]$ by assumption, from the latter we obtain
        \begin{gather}\label{filters aux2}
            \Diamond_{\mathbf{B}}(h(a_1),h(a_2),h(a_3))\in U.
        \end{gather}
        Let $T_{i}$ be the filter generated by $h[F_{Z_i}]$. Each $T_i$ is proper. Indeed, suppose that $0\in T_1$.
        Then, there exists
        $a_1\in F_{Z_1}$ such that $h(a_1)\leq 0$,
        and hence $h(a_1)=0$.
        Since $1\in F_{Z_2}$ and $1\in F_{Z_3}$,
        equation~\eqref{filters aux2}, together with $h(1)=1$
        and \eqref{MO1}, yields
        \[
            0=\Diamond_{\mathbf{B}}(0,1,1)
         =\Diamond_{\mathbf{B}}(h(a_1),h(1),h(1))\in U,
        \]
        contrary to the properness of the ultrafilter $U$.
        The same argument, placing $0$ in the second or third
        coordinate, shows that $T_2$ and $T_3$ are proper. It is immediate from (\ref{filters aux2}) and the monotonicity of $\Diamond_{\mathbf{B}}$ that
        \[
        T_1\times T_2\times T_3\subseteq \Diamond_{\mathbf{B}}^{-1}[U].
        \]
        Consequently, if we take $Y_i\defeq Y_{T_i}$, it is readily seen that $\vec{Y}\in R_{\Diamond_{\mathbf{B}}}(U)$ and $f[\vec{Y}]\subseteq \vec{Z}$. This concludes the proof.

        \smallskip

        (2) (a) $\Leftrightarrow$ (b).
        By the identities established in the proof of
        (1)(a) $\Leftrightarrow$ (b), the inclusion in (2)(b)
        is equivalent to
        \[
        \beta_{\mathbf{B}}\big(
        \Diamond_{\mathbf{B}}(h(a),h(b),h(c))
        \big)
        \subseteq
        \beta_{\mathbf{B}}\big(
        h(\Diamond_{\mathbf{A}}(a,b,c))
        \big)
        \]
        for every $a,b,c\in A$.
        Since $\beta_{\mathbf{B}}$ preserves and reflects order,
        this is equivalent to
        \[
        \Diamond_{\mathbf{B}}(h(a),h(b),h(c))
        \leq h(\Diamond_{\mathbf{A}}(a,b,c))
        \]
        for every $a,b,c\in A$, which is precisely the condition
        that $h$ be a PSI hemi-homomorphism.

        \smallskip

        (c) $\Rightarrow$ (b). Let us assume that $U\in R_{\Diamond_{\mathbf{B}}}^{-1}(\beta_{\mathbf{B}}(h(a)),\beta_{\mathbf{B}}(h(b)),\beta_{\mathbf{B}}(h(c)))$. By (c) we have that $f(U)\in R_{\Diamond_{\mathbf{A}}}^{-1}(f[\beta_{\mathbf{B}}(h(a))],f[\beta_{\mathbf{B}}(h(b))],f[\beta_{\mathbf{B}}(h(c))])$. Finally, from (\ref{Property Stone duality}) and the monotonicity of $R_{\Diamond_{\mathbf{A}}}^{-1}$, we can conclude that $f(U)\in R_{\Diamond_{\mathbf{A}}}^{-1}(\beta_{\mathbf{A}}(a),\beta_{\mathbf{A}}(b),\beta_{\mathbf{A}}(c))$. Hence, (b) holds.

        \smallskip

        (b) $\Rightarrow$ (c). Suppose that $(U,\vec{Y})\in R_{\Diamond_{\mathbf{B}}}$. Then, by the monotonicity of $R_{\Diamond_{\mathbf{B}}}^{-1}$, it is the case that (k) $U\in R_{\Diamond_{\mathbf{B}}}^{-1}(\beta_{\mathbf{B}}(b_1),\beta_{\mathbf{B}}(b_2),\beta_{\mathbf{B}}(b_3))$, for every $b_1,b_2,b_3\in B$, such that $Y_i\subseteq \beta_{\mathbf{B}}(b_i)$. In order to prove our claim, let $a_1,a_2,a_3\in A$ be such that $f[Y_i]\subseteq \beta_{\mathbf{A}}(a_i)$. Then, by (\ref{Property Stone duality}), it follows that $Y_i\subseteq \beta_{\mathbf{B}}(h(a_i))$. So, from (b) and (k) we can conclude that $(f(U), f[\vec{Y}])\in R_{\Diamond_{\mathbf{A}}}$, as required.
\end{proof}

Let $(X_1,\tau_1, R_1)$ and $(X_2, \tau_2, R_2)$ be two PSI-spaces, and let $f\colon X_1 \to X_2$ be a function. Consider the following conditions:

\begin{enumerate}[label=(Sp\arabic*),ref=Sp\arabic*]
    \item \label{Sp1} For any $U \in \CO(X_2)$, we have $f^{-1}[U] \in \CO(X_1)$.
    \item \label{Sp2} If $(x, \vec{Y}) \in R_1$, then $(f(x) , f[\vec{Y}]) \in R_2$, where
    \[
    f[Z] = \{ f(z) \mid z \in Z \} \quad \text{and} \quad f[\vec{Y}] = (f[Y_1], f[Y_2], f[Y_3])\,.
    \]

    \item \label{Sp3} If $(f(x) , \vec{Z}) \in R_2$, then there exists $\vec{Y} \in \mathscr{C}_0 (X_1)^3$ such that $\vec{Y} \in R_1(x)$ and $f[\vec{Y}] \subseteq \vec{Z}$.
\end{enumerate}

\begin{definition}\label{psi-spaces morphism}
    Let $f\colon X_1 \to X_2$ be a function between PSI-spaces $(X_1,\tau_1, R_1)$ and $(X_2, \tau_2, R_2)$. We say that $f$ is a \emph{semi-PSI map} if \eqref{Sp1} and \eqref{Sp3} hold. Further, $f$ is a \emph{hemi-PSI map} if \eqref{Sp1} and \eqref{Sp2} hold. Finally, we say that $f$ is a \emph{PSI map} if it is both a semi-PSI map and a hemi-PSI map.\QED
\end{definition}

    The identity function on a PSI-space acts as a semi-PSI (hemi-PSI or PSI) map, and when two semi-PSI (hemi-PSI or PSI) maps are composed functionally, the result is also a semi-PSI (hemi-PSI or PSI) map. This structure gives rise to three categories: the category $\mathsf{PSIsp_0}$, which consists of PSI-spaces and semi-PSI maps, the category $\mathsf{PSIsp_1}$, which consists of PSI-spaces and hemi-PSI maps, and the category $\mathsf{PSIsp_2}$ which consists of PSI-spaces and PSI maps. Similarly, the standard composition of two semi-homomorphisms (hemi-homomorphisms or homomorphisms) results in another semi-homomorphism (hemi-homomorphism or homomorphism), the identity function on a PSI-Algebra is itself a semi-homomorphism (hemi-homomorphism or homomorphism). Consequently, PSI-Algebras and their semi-homomorphisms form a concrete category, denoted by $\mathsf{PSI_0}$. Likewise, the category with PSI-Algebras as objects and hemi-homomorphisms as arrows is denoted by $\mathsf{PSI_1}$, and the category with PSI-Algebras as objects and homomorphisms as arrows is denoted by $\mathsf{PSI_2}$.
\medskip

The following proposition characterizes the isomorphisms in these three categories.
\begin{proposition}\label{isos in algebras}
    In each of the categories $\mathsf{PSI_0}$, $\mathsf{PSI_1}$ and $\mathsf{PSI_2}$, the isomorphisms are precisely the bijective PSI homomorphisms.
\end{proposition}
\begin{proof}
Let $h:(\mathbf{A},\Diamond_{\mathbf{A}})\to(\mathbf{B},\Diamond_{\mathbf{B}})$ be an isomorphism in $\mathsf{PSI_0}$, and let $g$ be its inverse in that category. Then $h$ is a Boolean isomorphism and $g=h^{-1}$ as a function.
Both $h$ and $g$ are semi-homomorphisms. Hence, for all $a,b,c\in A$,
\[
h(\Diamond_{\mathbf{A}}(a,b,c))
\leq \Diamond_{\mathbf{B}}(h(a),h(b),h(c)).
\]
Applying the semi-homomorphism inequality for $g$ to $h(a),h(b),h(c)$ gives
\[
g(\Diamond_{\mathbf{B}}(h(a),h(b),h(c)))
\leq \Diamond_{\mathbf{A}}(a,b,c).
\]
Since $h$ preserves order, applying $h$ to the latter inequality yields the reverse of the former. Thus equality holds, so $h$ is a PSI homomorphism. For $\mathsf{PSI_1}$, the same argument, with both inequalities reversed, yields the same conclusion. For $\mathsf{PSI_2}$, $h$ is a PSI homomorphism by definition.

Conversely, let $h$ be a bijective PSI homomorphism. Its inverse $g$ is also a PSI homomorphism. Therefore, $h$ and $g$ are morphisms in each of the three categories and are inverse to one another.
\end{proof}

In light of Definition \ref{psi-spaces morphism}, the following result is a direct consequence of Lemma \ref{morphisms to maps}.

\begin{lemma}\label{char morph to spaces}
    Let $(\mathbf{A},\Diamond_{\mathbf{A}})$ and $(\mathbf{B},\Diamond_{\mathbf{B}})$ be two PSI-Algebras and let $h\colon\mathbf{A} \to \mathbf{B}$ be a homomorphism of Boolean algebras. Let us consider the map $h^{-1}\colon\Uf(\mathbf{B})\to \Uf(\mathbf{A})$. Then, the following hold:
    \begin{enumerate}
        \item $h$ is a PSI semi-homomorphism if and only if $h^{-1}$ is a semi-PSI map.
        \item $h$ is a PSI hemi-homomorphism if and only if $h^{-1}$ is a hemi-PSI map.
        \item $h$ is a PSI homomorphism if and only if $h^{-1}$ is a PSI map.
    \end{enumerate}
\end{lemma}

An adaptation of the proof of \citep[Theorem 18]{Celani2009} yields the following result.

\begin{lemma}\label{char spaces to morph}
    Let $f\colon X_1 \to X_2 $ be a continuous mapping between the PSI-spaces $(X_1,\tau_1,R_1)$ and $(X_2,\tau_2,R_2)$. Consider the function $f^{*}\colon \CO(X_2) \to \CO(X_1)$ defined by $ f^{*}(U)\defeq f^{-1}[U]$. Then, the following hold:
    \begin{enumerate}
        \item $ f $ is a semi-PSI map if and only if $ f^*$ is a PSI semi-homomorphism.
        \item $ f $ is a hemi-PSI map if and only if $ f^*$ is a PSI hemi-homomorphism.
        \item $ f $ is a PSI map if and only if $ f^*$ is a PSI homomorphism.
    \end{enumerate}
\end{lemma}

Now, we characterize isomorphisms in $\mathsf{PSIsp}_0$, $\mathsf{PSIsp}_1$ and $\mathsf{PSIsp}_2$.

\begin{proposition}\label{isos in spaces}
    In each of the categories $\mathsf{PSIsp}_0$, $\mathsf{PSIsp}_1$ and $\mathsf{PSIsp}_2$, the isomorphisms are precisely the bijective PSI maps.
\end{proposition}
\begin{proof}
    Let $f \colon X_1 \to X_2$ be a $\mathsf{PSIsp}_0$-isomorphism between $(X_1,\tau_1,R_1)$ and $(X_2,\tau_2,R_2)$ with inverse $g \colon X_2 \to X_1$. Then $f$ is a semi-PSI map and $g$ is a semi-PSI map. By Lemma~\ref{char spaces to morph}(1), the dual map $f^* \colon \CO(X_2) \to \CO(X_1)$ is a PSI semi-homomorphism, and $g^* = (f^*)^{-1}$ is also a PSI semi-homomorphism. By Proposition~\ref{isos in algebras}, it follows that $f^*$ is a $\mathsf{PSI_0}$-isomorphism. Thus, by Lemma~\ref{char spaces to morph}(3), $f$ is a PSI map. Analogous arguments apply to $\mathsf{PSIsp}_1$ and $\mathsf{PSIsp}_2$.

    Conversely, if $f \colon X_1 \to X_2$ is a bijective PSI map, then by Lemma~\ref{char spaces to morph}(3) its dual $f^{\ast} \colon \CO(X_2) \to \CO(X_1)$ is a PSI isomorphism. By Proposition~\ref{isos in algebras}, $f^{\ast}$ is an isomorphism in $\mathsf{PSI}_j$ for every $j \in \{0,1,2\}$, meaning its inverse $(f^{\ast})^{-1}$ is a morphism in that category. Since $f$ is bijective, well-known set-theoretic arguments imply that $(f^{\ast})^{-1} = (f^{-1})^{\ast}$, and thus $(f^{-1})^{\ast}$ is a PSI homomorphism. Applying Lemma~\ref{char spaces to morph} to $f^{-1}$, we conclude that $f^{-1}$ is a PSI map, and hence a morphism in $\mathsf{PSIsp}_j$. Since $f \circ f^{-1} = \operatorname{id}_{X_2}$ and $f^{-1} \circ f = \operatorname{id}_{X_1}$, it follows that $f$ is an isomorphism in each of the three categories.
\end{proof}

For a category $\mathcal{C}$, we denote by
$\mathcal{C}^{\mathrm{op}}$ its opposite category, which has the same objects as $\mathcal{C}$ and all arrows reversed.

Let $j\in \{0,1,2\}$. By Theorem \ref{Psi-algebras as ultrafilter frames} and Lemma \ref{char morph to spaces}, we are able to define functors
\[
G_j\colon \mathsf{PSI_j}^{\mathrm{op}} \to \mathsf{PSIsp_j}
\]
as follows
\[
\begin{array}{rcl}
    (\mathbf{B}, \Diamond) & \mapsto & (\Uf(\mathbf{B}), \tau_{\mathbf{B}}, R_{\Diamond}), \\
    h & \mapsto & h^{-1}.
\end{array}
\]

Analogously, by Proposition \ref{space to algebra} and Lemma \ref{char spaces to morph} the assignments
\[
\begin{array}{rcl}
   (X, \tau, R)  & \mapsto & (\CO(X), \Diamond_R),  \\
    f & \mapsto & f^{\ast}
\end{array}
\]

also define functors
\[
H_j\colon \mathsf{PSIsp}_j \to \mathsf{PSI}_j^{\mathrm{op}}.
\]

We are intended to show that the functors $G_j$ and $H_j$ with $j\in \{0,1,2\}$ yield an equivalence of categories. To do so, we will need some technical results.
\medskip

For each PSI-algebra $\mathfrak{A} = (\mathbf{A}, \Diamond)$, let us consider the standard Stone evaluation map
\[
\eta_{\mathfrak{A}} \colon \mathbf{A} \to \CO(\Ult(\mathbf{A})), \quad a \mapsto \beta_{\mathbf{A}}(a).
\]
By classical Stone duality, $\eta_{\mathfrak{A}}$ is an isomorphism of Boolean algebras.
\begin{proposition}\label{prop:unit-isomorphism}
    Let $\mathfrak{A} = (\mathbf{A}, \Diamond_{\mathbf{A}})$ be a PSI-algebra. Then, for every $j\in \{0,1,2\}$, the Stone map $\eta_{\mathfrak{A}} \colon \mathbf{A} \to \CO(\Ult(\mathbf{A}))$ is a $\mathsf{PSI}_j$-isomorphism. Moreover, for every $j\in \{0,1,2\}$, such an isomorphism is natural.
\end{proposition}
\begin{proof}
    By classical Stone duality, $\eta_{\mathfrak{A}}$ is a Boolean isomorphism, so in particular it is bijective. By Proposition \ref{isos in algebras}, to show our claim, it is enough to prove that it preserves $\Diamond_{\mathbf{A}}$, that is, for all $a,b,c \in A$,
    \[
    \eta_{\mathfrak{A}}(\Diamond_{\mathbf{A}}(a,b,c)) = \Diamond_{R_{\Diamond_{\mathbf{A}}}}(\eta_{\mathfrak{A}}(a), \eta_{\mathfrak{A}}(b), \eta_{\mathfrak{A}}(c)).
    \]
    By definition of the dual structure and Proposition~\ref{psi-frames are m-frames}, we have
    \[
    \Diamond_{R_{\Diamond_{\mathbf{A}}}}(\beta_{\mathbf{A}}(a), \beta_{\mathbf{A}}(b), \beta_{\mathbf{A}}(c)) = R_{\Diamond_{\mathbf{A}}}^{-1}(\beta_{\mathbf{A}}(a), \beta_{\mathbf{A}}(b), \beta_{\mathbf{A}}(c)).
    \]
    By Proposition~\ref{stone-relation lemma}(1), this set equals $\beta_{\mathbf{A}}(\Diamond_{\mathbf{A}}(a,b,c))$, which is precisely $\eta_{\mathfrak{A}}(\Diamond_{\mathbf{A}}(a,b,c))$. To establish naturality, let $\mathfrak{B}=(\mathbf{B},\Diamond_{\mathbf{B}})$ be a PSI-algebra, $j\in \{0,1,2\}$ and let $h \colon \mathfrak{A} \to \mathfrak{B}$ be a $\mathsf{PSI}_j$-morphism. Observe that since $\eta$ coincides with the unit of the classical Stone duality, then the commutativity of the following diagram:
    \[
    \xymatrix{
        \mathfrak{A} \ar[r]^{\eta_{\mathfrak{A}}} \ar[d]_{h} & H_jG_j(\mathfrak{A})
        \ar[d]^{H_jG_j(h)} \\
        \mathfrak{B} \ar[r]_{\eta_{\mathfrak{B}}} & H_jG_j(\mathfrak{B})
    }
    \]
  reduces directly to the commutativity on the underlying Boolean reducts. This shows that $\eta$ is a natural isomorphism.
\end{proof}

For each PSI-space $\mathcal{X} = (X, \tau, R)$, let us consider the evaluation map
\[
\varepsilon_{\mathcal{X}} \colon X \to \Ult(\CO(X)), \quad x \mapsto \{U \in \CO(X) \mid x \in U\}.
\]
By classical Stone duality, $\varepsilon_{\mathcal{X}}$ is a homeomorphism of Boolean spaces.

\begin{proposition}\label{prop:counit-isomorphism}
    Let $\mathcal{X} = (X, \tau, R)$ be a PSI-space. Then, for every $j \in \{0,1,2\}$, the evaluation map $\varepsilon_{\mathcal{X}} \colon X \to \Ult(\CO(X))$ is a $\mathsf{PSIsp}_j$-isomorphism. Moreover, for every $j \in \{0,1,2\}$ such an isomorphism is natural.
\end{proposition}
\begin{proof}
    By classical Stone duality, $\varepsilon_{\mathcal{X}}$ is a homeomorphism, so in particular it is a bijection satisfying condition \eqref{Sp1} for both $\varepsilon_{\mathcal{X}}$ and its inverse.

    To show that $\varepsilon_{\mathcal{X}}$ is a hemi-PSI map (condition \eqref{Sp2}), assume $(x, \vec{Y}) \in R$. We need to show $(\varepsilon_{\mathcal{X}}(x), \varepsilon_{\mathcal{X}}[\vec{Y}]) \in R_{\Diamond_R}$. By definition of $R_{\Diamond_R}$, this is equivalent to showing that for all $A_1, A_2, A_3 \in \CO(X)$ such that $\varepsilon_{\mathcal{X}}[Y_i] \subseteq \beta(A_i)$ for each $i \in \{1,2,3\}$, we have $\Diamond_R(A_1, A_2, A_3) \in \varepsilon_{\mathcal{X}}(x)$, which means $x \in \Diamond_R(A_1, A_2, A_3) = R^{-1}(A_1, A_2, A_3)$. Since $\varepsilon_{\mathcal{X}}[Y_i] \subseteq \beta(A_i)$ implies $Y_i \subseteq A_i$, by the monotonicity of $R^{-1}$ and $(x, \vec{Y}) \in R$, this inclusion holds.

    To show that $\varepsilon_{\mathcal{X}}$ is a semi-PSI map (condition \eqref{Sp3}), let $(\varepsilon_{\mathcal{X}}(x), \vec{Z}) \in R_{\Diamond_R}$ for some basic closed sets $\vec{Z} = (\beta(A_1), \beta(A_2), \beta(A_3))$. Then $x \in R^{-1}(A_1, A_2, A_3)$, so $(x, A_1, A_2, A_3) \in R$. By choosing $\vec{Y} = (A_1, A_2, A_3) \in \mathscr{C}_0(X)^3$, we obtain $\vec{Y} \in R(x)$ and $\varepsilon_{\mathcal{X}}[\vec{Y}] \subseteq \vec{Z}$.

    Hence, $\varepsilon_{\mathcal{X}}$ is a bijective PSI map. By Proposition \ref{isos in spaces}, it is a $\mathsf{PSIsp}_j$-isomorphism for every $j \in \{0,1,2\}$.

To establish naturality, let $f \colon \mathcal{X}_1 \to \mathcal{X}_2$ be a $\mathsf{PSIsp}_j$-morphism between the PSI-spaces $\mathcal{X}_1 = (X_1, \tau_1, R_1)$ and $\mathcal{X}_2 = (X_2, \tau_2, R_2)$. We need to show that the following diagram commutes:
    \[
    \xymatrix{
        X_1 \ar[r]^{\varepsilon_{\mathcal{X}_1}} \ar[d]_{f} & G_jH_j(\mathcal{X}_1) \ar[d]^{G_jH_j(f)} \\
        X_2 \ar[r]_{\varepsilon_{\mathcal{X}_2}} & G_jH_j(\mathcal{X}_2)
    }
    \]
    Indeed, since $\varepsilon$ coincides with the counit of the Stone duality, the commutativity of this diagram reduces directly to the underlying Boolean space reducts, making the naturality of $\varepsilon$ immediate.
\end{proof}

As a consequence of Propositions \ref{prop:unit-isomorphism} and \ref{prop:counit-isomorphism}, we obtain the following

\begin{theorem}\label{dualidad PSI and PSISp}
The categories $\mathsf{PSI_0}$, $\mathsf{PSI_1}$ and $\mathsf{PSI_2}$ are dually equivalent to $\mathsf{PSIsp}_0$, $\mathsf{PSIsp}_1$, and $\mathsf{PSIsp}_2$, respectively.
\end{theorem}

Theorem \ref{dualidad PSI and PSISp} establishes the dual equivalences between the algebraic and the topological sides of the theory of PSI-Algebras. In view of this correspondence, we can now focus on the case of Extended Contact Algebras. To that end, we require the introduction of a special class of PSI-spaces, namely those for which the relational component exhibits a form of totality, as defined below.

\begin{definition}\label{EC-spaces}
    We say that a PSI-space $(X,\tau, R)$ is \emph{total} if $X\neq \emptyset$ and for every $\vec{Y}\in \mathscr{C}_0(X)^3$, the following condition holds
    \begin{enumerate}[label=(T),ref=T]
    \item \label{T} If there exists $x\in X$ such that $(x,\vec{Y})\in R$, then $(y,\vec{Y})\in R$, for every $y\in X$.
\end{enumerate}
Equivalently, \eqref{T} says that $R^{-1}(\vec{Y})\in\{X,\emptyset\}$, for every $\vec{Y}\in \mathscr{C}_0(X)^3$.\QED
\end{definition}

\begin{remark}\label{rem: total spaces}
Let $(\mathbf{B}, \vdash)$ be an ECA and let $(\mathbf{B}, \Diamond_{\vdash})$ be its associated non-degenerate relational PSI-Algebra. Since $\mathbf{B}$ is non-degenerate, by classical Stone duality, $\Uf(\mathbf{B})\neq \emptyset$. Thus, given $\vec{Y}\in \mathscr{C}_{0}(\Uf(\mathbf{B}))^3$, the relational nature of $(\mathbf{B}, \Diamond_{\vdash})$ ensures that $R_{\Diamond_{\vdash}}^{-1}(\vec{Y})\in \{\Uf(\mathbf{B}),\emptyset\}$, meaning that $(\Uf(\mathbf{B}),\tau_{\mathbf{B}}, R_{\Diamond_{\vdash}})$ is indeed a total PSI-space. Moreover, by Theorem \ref{dualidad PSI and PSISp}, $(\mathbf{B}, \Diamond_{\vdash})$ is isomorphic to $\left(\CO(\Uf(\mathbf{B})), \Diamond_{R_{\diamond_{\vdash}}}\right)$.\QED
\end{remark}

As an immediate consequence of Theorem \ref{Psi-algebras as ultrafilter frames}, Proposition \ref{space to algebra}, Proposition \ref{stone-relation lemma}, Definition \ref{EC-spaces} and Remark \ref{rem: total spaces} we obtain

\begin{proposition}
    The following hold:
    \begin{enumerate}
        \item $(\mathbf{B},\vdash)$ is an ECA iff $(\Uf(\mathbf{B}),\tau_{\mathbf{B}},  R_{\diamond_{\vdash}})$ is a total PSI-space.
        \item $(X, \tau, R)$ is a total PSI-space if and only if $(\CO(X), \Diamond_{R})$ is a relational PSI-Algebra.
    \end{enumerate}
\end{proposition}

Let $(\mathbf{A},\vdash_{\mathbf{A}})$ and $(\mathbf{B},\vdash_{\mathbf{B}})$ be two ECAs. We say that a homomorphism of Boolean algebras $h\colon\mathbf{A}\to \mathbf{B}$ is a \emph{reflecting ECA-morphism} if, for every $a,b,c\in A$:
\[(h(a),h(b)) \vdash_{\mathbf{B}} h(c)\; \Rightarrow\; (a,b) \vdash_{\mathbf{A}} c\,.\]
On the other hand, if
\[(a,b) \vdash_{\mathbf{A}} c\; \Rightarrow\; (h(a),h(b)) \vdash_{\mathbf{B}} h(c)\]
holds, then we say that $h$ is a \emph{preserving ECA-morphism}.Furthermore, if $h$ is both a reflecting and a preserving ECA-morphism, then we say that $h$ is a \emph{similarity ECA-morphism}.

\begin{proposition}\label{morphisms EC and PSI}
    Let $(\mathbf{A}_1,\vdash_{1})$ and $(\mathbf{A}_2,\vdash_{2})$ be two ECAs and let $(\mathbf{A}_i, \Diamond_{\vdash_i})$ be their associated non-degenerate relational PSI-Algebras, for $i\in \{1,2\}$. If $h\colon A_1 \rightarrow A_2$ is a  Boolean homomorphism, then the following hold:
    \begin{enumerate}
        \item $h$ is a reflecting ECA-morphism iff $h$ is a PSI semi-homomorphism.
        \item $h$ is a preserving ECA-morphism iff $h$ is a PSI hemi-homomorphism.
        \item $h$ is a similarity ECA-morphism iff $h$ is a PSI homomorphism.
    \end{enumerate}
\end{proposition}

We write $\mathsf{EC}_0$, $\mathsf{EC}_1$ and $\mathsf{EC}_2$ for the categories whose objects are ECAs and whose morphisms are respectively, reflecting, preserving, and similarity ECA-morphisms. In addition, we denote by $\mathsf{tPSIsp}_0$, $\mathsf{tPSIsp}_1$ and $\mathsf{tPSIsp}_2$ the full subcategories of $\mathsf{PSIsp}_0$, $\mathsf{PSIsp}_1$ and $\mathsf{PSIsp}_2$ whose objects are total PSI-spaces.
\medskip

As an immediate consequence of propositions \ref{morphisms EC and PSI} and \ref{isos in algebras} we obtain
\begin{corollary}
    In each of the categories $\mathsf{EC}_0$, $\mathsf{EC}_1$ and $\mathsf{EC}_2$, the isomorphisms are precisely the bijective similarity ECA-morphisms.
\end{corollary}

Building upon these results, we now establish the final theorem, which demonstrates how the duality theory constructed here provides a complete structural identification of Extended Contact Algebras within our overarching framework.

\begin{theorem}\label{dualidad EC and tPSISp}
The categories $\mathsf{EC}_0$, $\mathsf{EC}_1$ and $\mathsf{EC}_2$ are dually equivalent to $\mathsf{tPSIsp}_0$, $\mathsf{tPSIsp}_1$ and $\mathsf{tPSIsp}_2$, respectively.
\end{theorem}
\begin{proof}
    Let $j\in \{0,1,2\}$ and consider the assignments
\[
L_j\colon\mathsf{EC_j}^{\mathrm{op}} \to \mathsf{tPSIsp_j}
\]
defined as follows
\[
\begin{array}{rcl}
    (\mathbf{B}, \vdash) & \mapsto &  (\Uf(\mathbf{B}), \tau_{\mathbf{B}}, R_{\Diamond_{\vdash}}) \\
    h & \mapsto & h^{-1}.
\end{array}
\]

And also consider the assignments
\[
M_j\colon\mathsf{tPSIsp}_j \to \mathsf{EC}_j^{\mathrm{op}}
\]
defined by
\[
\begin{array}{rcl}
    (X, \tau, R) & \mapsto & (\CO(X), \vdash_{\Diamond_R})  \\
    f & \mapsto & f^{-1}.
\end{array}
\]
Notice that from Lemma \ref{extended contact are relational pseudo-iinference algebras} and Theorem \ref{dualidad PSI and PSISp} it follows immediately that every pair $L_j$ and $M_j$ forms a pair of functors that result in the desired equivalence.
\end{proof}

As a final remark, we point out that a promising direction for future research is to investigate to what extent the representation and duality results developed here extend to Boolean algebras equipped with normal monotone $m$-ary operators that are additive in $n$ specified coordinates, where $1\leq n\leq m$. A related question is whether analogous algebraic and topological constructions can be developed for algebraizations of the $S$-algebras considered in \citep{Vakarelov-AMBOSQ}. We leave these extensions for future investigation.

\bibliographystyle{apalike}

\providecommand{\noop}[1]{}

\end{document}